\documentclass[review]{elsarticle}
\usepackage{srcltx}
\usepackage{eurosym}
\usepackage{mathtools}
\usepackage{amsmath}
\usepackage{amsfonts}
\usepackage{amssymb}
\usepackage{amsthm}
\usepackage{graphicx}
\usepackage{mathrsfs}
\usepackage{xcolor}
\usepackage{exscale}
\usepackage{latexsym}
\usepackage{multicol}

\usepackage[backref=page,colorlinks,plainpages=true,pdfpagelabels,
hypertexnames=true,colorlinks=true,pdfstartview=FitV,linkcolor=blue,
citecolor=red,urlcolor=black]{hyperref}
\PassOptionsToPackage{unicode}{hyperref}
\PassOptionsToPackage{naturalnames}{hyperref}
\usepackage{enumerate}
\usepackage[shortlabels]{enumitem}
\usepackage{bookmark}
\usepackage{wasysym}
\usepackage{esint}
\usepackage[ddmmyyyy]{datetime}
\usepackage[margin=2cm]{geometry}
\makeatletter
\g@addto@macro\normalsize{%
  \setlength\abovedisplayskip{2pt}
  \setlength\belowdisplayskip{2pt}
  \setlength\abovedisplayshortskip{4pt}
  \setlength\belowdisplayshortskip{4pt}
}
\numberwithin{equation}{section}
\everymath{\displaystyle}
\usepackage[capitalize,nameinlink]{cleveref}
\crefname{section}{Section}{Sections}
\crefname{subsection}{Subsection}{Subsections}
\crefname{condition}{Condition}{Conditions}
\crefname{hypothesis}{Hypothesis}{Hypothesis}
\crefname{assumption}{Assumption}{Assumptions}
\crefname{lemma}{Lemma}{Lemmas}
\crefname{claim}{Claim}{Claims}
\crefname{remark}{Remark}{Remarks}

\crefformat{equation}{\textup{#2(#1)#3}}
\crefrangeformat{equation}{\textup{#3(#1)#4--#5(#2)#6}}
\crefmultiformat{equation}{\textup{#2(#1)#3}}{ and \textup{#2(#1)#3}}
{, \textup{#2(#1)#3}}{, and \textup{#2(#1)#3}}
\crefrangemultiformat{equation}{\textup{#3(#1)#4--#5(#2)#6}}%
{ and \textup{#3(#1)#4--#5(#2)#6}}{, \textup{#3(#1)#4--#5(#2)#6}}%
{, and \textup{#3(#1)#4--#5(#2)#6}}

\Crefformat{equation}{#2Equation~\textup{(#1)}#3}
\Crefrangeformat{equation}{Equations~\textup{#3(#1)#4--#5(#2)#6}}
\Crefmultiformat{equation}{Equations~\textup{#2(#1)#3}}{ and \textup{#2(#1)#3}}
{, \textup{#2(#1)#3}}{, and \textup{#2(#1)#3}}
\Crefrangemultiformat{equation}{Equations~\textup{#3(#1)#4--#5(#2)#6}}%
{ and \textup{#3(#1)#4--#5(#2)#6}}{, \textup{#3(#1)#4--#5(#2)#6}}%
{, and \textup{#3(#1)#4--#5(#2)#6}}

\crefdefaultlabelformat{#2\textup{#1}#3}
\newtheorem{theorem}{Theorem}[section]
\newtheorem{lemma}[theorem]{Lemma}
\newtheorem{claim}[theorem]{Claim}
\newtheorem{corollary}[theorem]{Corollary}
\newtheorem{proposition}[theorem]{Proposition}

\newtheorem{definition}[theorem]{Definition}
\newtheorem{remark}[theorem]{Remark}        

\numberwithin{equation}{section}
\def\aa{\mathcal{A}}

\newcommand{\pa}{\partial}
\newcommand{\ddt}[1]{\frac{d#1}{dt}}

\newcommand{\tQ}{\tilde{Q}}

\newcommand{\trho}{\tilde{\rho}}

\newcommand{\ttht}{\tilde{\tht}}

\makeatletter
\newcommand{\vo}{\vec{o}\@ifnextchar{^}{\,}{}}
\makeatother
\def\Yint#1{\mathchoice
    {\YYint\displaystyle\textstyle{#1}}%
    {\YYint\textstyle\scriptstyle{#1}}%
    {\YYint\scriptstyle\scriptscriptstyle{#1}}%
    {\YYint\scriptscriptstyle\scriptscriptstyle{#1}}%
      \!\iint}
\def\YYint#1#2#3{{\setbox0=\hbox{$#1{#2#3}{\iint}$}
    \vcenter{\hbox{$#2#3$}}\kern-.50\wd0}}
\def\longdash{-\mkern-9.5mu-} 
\def\tiltlongdash{\rotatebox[origin=c]{18}{$\longdash$}}
\def\fiint{\Yint\tiltlongdash}
\def\Xint#1{\mathchoice
    {\XXint\displaystyle\textstyle{#1}}%
    {\XXint\textstyle\scriptstyle{#1}}%
    {\XXint\scriptstyle\scriptscriptstyle{#1}}%
    {\XXint\scriptscriptstyle\scriptscriptstyle{#1}}%
      \!\int}
\def\XXint#1#2#3{{\setbox0=\hbox{$#1{#2#3}{\int}$}
    \vcenter{\hbox{$#2#3$}}\kern-.50\wd0}}
\def\hlongdash{-\mkern-13.5mu-}
\def\tilthlongdash{\rotatebox[origin=c]{18}{$\hlongdash$}}
\def\hint{\Xint\tilthlongdash}
\makeatletter
\def\namedlabel#1#2{\begingroup
   \def\@currentlabel{#2}%
   \label{#1}\endgroup
}
\makeatother
\makeatletter
\newcommand{\rmh}[1]{\mathpalette{\raisem@th{#1}}}
\newcommand{\raisem@th}[3]{\hspace*{-1pt}\raisebox{#1}{$#2#3$}}
\makeatother

\newcommand{\lsb}[2]{#1_{\rmh{-3pt}{#2}}}
\newcommand{\lsbo}[2]{#1_{\rmh{-1pt}{#2}}}

\newcommand{\redref}[2]{\texorpdfstring{\protect\hyperlink{#1}{\textcolor{black}{(}\textcolor{red}{#2}\textcolor{black}{)}}}{}}
\newcommand{\redlabel}[2]{\hypertarget{#1}{\textcolor{black}{(}\textcolor{red}{#2}\textcolor{black}{)}}}

\newcommand{\descitem}[2]{\item[{(#1):}]\label{#2}}
\newcommand{\descref}[2]{\hyperref[#1]{\textnormal{\textcolor{black}{(}\textcolor{blue}{\bf #2}\textcolor{black}{)}}}}
\newcommand{\ditem}[2]{\item[#1] \label{#2}}
\newcommand{\dref}[2]{\hyperref[#1]{\textcolor{black}{(}\textcolor{blue}{\bf #2}\textcolor{black}{)}}}

\newcommand{\tp}{\tilde{p}}
\newcommand{\tq}{\tilde{q}}

\newcommand{\tw}{\tilde{w}}

\newcommand{\tz}{\tilde{z}}

\newcommand\RR{\mathbb{R}}

\newcommand\ZZ{\mathbb{Z}}

\newcommand{\al}{\alpha}
\newcommand{\be}{\beta}
\newcommand{\ga}{\gamma}
\newcommand{\de}{\delta}
\newcommand{\ve}{\varepsilon}

\newcommand{\tht}{\theta}
\newcommand{\ep}{\epsilon}
\newcommand{\ka}{\kappa}

\newcommand{\la}{\lambda}

\newcommand{\Om}{\Omega}

\DeclareMathOperator{\dv}{div}
\DeclareMathOperator{\spt}{spt}

\DeclareMathOperator{\loc}{loc}

\newcommand{\iprod}[2]{\langle #1 \ ,  #2\rangle}
\newcommand{\abs}[1]{\left| #1\right|}

\newcommand{\lbr}[1][(]{\left#1}
\newcommand{\rbr}[1][)]{\right#1}

\newcommand{\avgs}[2]{\lsbo{\lbr #1 \rbr}{#2}}

\newcommand{\txt}[1]{\qquad \text{#1} \qquad}

\newcommand{\vkn}[1][n+1]{(v-k_{#1})_+}

\newcommand{\C}{\mathcal{C}}

\newcommand{\data}{\{C_0,C_1,N,p\}}
\newcounter{whitney}
\refstepcounter{whitney}

\newcounter{ineqcounter}
\refstepcounter{ineqcounter}
\makeatletter
\def\ps@pprintTitle{%
\let\@oddhead\@empty
\let\@evenhead\@empty
\def\@oddfoot{}%
\let\@evenfoot\@oddfoot}
\makeatother
\usepackage{setspace}
\usepackage[titletoc,toc,page]{appendix}

\usepackage{pgf,tikz}
\usetikzlibrary{arrows,patterns}
\usetikzlibrary{decorations.pathreplacing}

\begin{document}

\begin{frontmatter}

\title{ Unified approach to   $C^{1,\alpha}$ regularity for quasilinear parabolic equations }

\author[myaddress]{Karthik Adimurthi\tnoteref{thanksfirstauthor}}
\cortext[mycorrespondingauthor]{Corresponding author}
\ead{karthikaditi@gmail.com and kadimurthi@tifrbng.res.in}

\author[myaddress]{Agnid Banerjee\tnoteref{thankssecondauthor}}
\ead{agnidban@gmail.com and agnid@tifrbng.res.in}

\tnotetext[thanksfirstauthor]{Supported by the Department of Atomic Energy,  Government of India, under
project no.  12-R\&D-TFR-5.01-0520}
\tnotetext[thankssecondauthor]{Supported in part by SERB Matrix grant MTR/2018/000267 and by Department of Atomic Energy,  Government of India, under
project no.  12-R \& D-TFR-5.01-0520}

\address[myaddress]{Tata Institute of Fundamental Research, Centre for Applicable Mathematics,Bangalore, Karnataka, 560065, India}

\begin{abstract}
In this paper, we are interested in obtaining a unified approach to  $C^{1,\alpha}$ estimates for weak solutions of quasilinear parabolic equations, the prototype example being
\[
    u_t - \dv (|\nabla u|^{p-2} \nabla u) = 0.
\]
without having to consider the singular and degenerate cases separately. This is achieved via  a new scaling and a delicate adaptation of the covering argument developed by E.~DiBenedetto and A.~Friedman.


\end{abstract}


\end{frontmatter}
\begin{singlespace}
\tableofcontents
\end{singlespace}
\section{Introduction}
\label{section1}
In this paper, we study gradient regularity of weak solutions of equations of the form 
\begin{equation}
    \label{main_eqn}
    u_t - \dv \aa(\nabla u) = 0,
\end{equation}
where the nonlinearity is modelled on the  $p$-Laplace operator. Moreover, we assume  $\aa:\RR^N \rightarrow \RR^N$ satisfying the following structure conditions for some $0 \leq s \leq 1$:
\begin{equation}
    \label{ellipticity}
    \left\{ \begin{array}{l}
                |\aa(\zeta)| + | \aa'(\zeta)| (|\zeta|^2+ s^2)^{\frac12} \leq C_1 (|\zeta|^2 + s^2)^{\frac{p-1}{2}}, \\
                \iprod{\aa'(\zeta)\eta}{\eta} \geq C_0 (|\zeta|^2 + s^2)^{\frac{p-2}{2}}|\eta|^2,
            \end{array}\right.
\end{equation}
where we have denoted $\aa'(\zeta) := \frac{d\aa(\zeta)}{d\zeta}$. 
\begin{remark}
    We believe it should be possible to  consider more general structures of the form $\aa(x,t,\zeta)$ with appropriate assumptions made regarding the behaviour with respect to $x,t$.  These assumptions are given in \cite[Section 1-(ii) of Chapter VIII]{DB93}. However, we have not explored this possibility and it would be interesting to see if it is indeed the case.
\end{remark}

Since we prove two main results, one being Lipschitz regularity and the other being $C^{1,\alpha}$ regularity, we shall discuss them separately in the following subsections. Concerning the $C^{1, \alpha}$  results in the elliptic case, we refer the reader to the well known papers \cite{Di}, \cite{Le} and \cite{To}.

\subsection{Discussion about Lipschitz regularity}

   There has been two approaches to study Lipschitz regularity for quasilinear parabolic equations, one developed in \cite{DBF} that makes use of Moser's iteration to first show that the solution $u \in W^{1,q}_{\loc}$ for any $q\in [1,\infty)$ and then adapting the DeGiorgi iteration techniques to obtain $u \in W^{1,\infty}_{\loc}$. The other approach was developed in \cite{Choe}, where only the Moser iteration was used to obtain Lipschitz regularity starting from $u\in W^{1,p}_{\loc}$ solutions.
   It is important to note that all the Lipschitz estimates have been obtained for the singular and degenerate cases using different techniques. 

   In this paper, we give the details for obtaining uniform Lipschitz estimates for the prototype equation
   \begin{equation}\label{prot}
       u_t - \dv(|\nabla u|^{p-2} \nabla u)=0,
   \end{equation}
    since the energy estimates needed for this are available in \cite[Chapter VIII]{DB93} and the calculations are more illustrative, noting that extending it to more general operators can be done in a standard way. Moreover, in this paper, we start with the a priori assumption that $u \in W^{1,q}_{\loc}$ for any $q\in [1,\infty)$ and then prove $u \in W^{1,\infty}_{\loc}$ by adapting the DeGiorgi iteration as in \cite{DBF}.  This is because, we are interested in obtaining a unified quantitative Lipschitz estimate which is useful in many applications, an important example being to study  Calderon-Zygmund type estimates.
    
    The novelty here is that we obtain Lipschitz estimates without having to differentiate between singular and degenerate cases.
   Our approach to proving this result  is  based on De Giorgi's approach by suitably adapting the ideas developed in \cite{APAMS} and combining it with \cite{DBF}.  The main idea is to note that when applying Sobolev embedding, there is still some available flexibility than what is used in \cite{DBF}, and it is this extra information that we make use of in order to obtain uniform Lipschitz estimates.  We would like to refer the reader to the recent interesting paper \cite{Cristiana} where a unified approach to Lipschitz estimate has been developed   which is instead based on Moser iteration. 
\subsection{Discussion about \texorpdfstring{$C^{1,\alpha}$}. regularity} 

   There has been a long history regarding the developed of $C^{1,\alpha}$ theory for quasilinear parabolic equations and we refer the reader to the detailed chronological development given in \cite{DB93}. The highly influential method using intrinsic scaling was developed in \cite{DBF}, \cite{DBF1} ( see also \cite{W})  to prove $C^{1,\alpha}$ regularity for quasilinear parabolic systems. This required studying the solution on cylinders whose size intrinsically depended on the solution itself. Though it must be noted that their method (see \cite[Chapter IX]{DB93}) actually does not need any intrinsic geometry, even though it is not explicitly written that way. To be consistent with existing notation, we shall also use the term 'intrinsic scaling' to denote the geometry considered in \cite{DBF}.   Subsequently, in  \cite{KM}, building upon all the previous ideas, the authors obtained a clean and elegant way of proving $C^{1,\alpha}$ H\"older regularity for quasilinear parabolic equations.
   
    \begin{remark}
        In all the discussion about $C^{1,\alpha}$ regularity, we start with the assumption that the solution is a priori Lipschitz continuous and restrict our discussion to this situation. Moreover, by $C^{1,\alpha}$ regularity, we mean that the spatial derivative $\nabla_xu$ is H\"older continuous.
    \end{remark}

  It is important to note that the standard  technique of proving $C^{1,\alpha}$ regularity for quasilinear parabolic equations required using intrinsic geometry and studying singular and degenerate cases separately. \emph{In this paper, we develop new scaling and suitably adapt the covering argument from \cite[Chapter IX]{DB93} using which  we prove $C^{1,\alpha}$ regularity without having to  differentiate the singular and degenerate cases.}  In closing,  we  refer to the recent interesting  work  \cite{ITS}  where instead using non-divergence techniques,  a   unified approach to $C^{1,\alpha}$ regularity  is developed for a  certain family of equations which includes the parabolic p-Laplace equation.  However  the class of equations studied in \cite{ITS} doesn't  cover  the general structure   \eqref{main_eqn} treated in the present work.



\section*{Acknowledgments}
The first author thanks Adi Adimurthi and Ugo Gianazza for many helpful discussions and suggestions.

\section{Preliminaries}
\label{section2}

In this section, we shall collect all the preliminary material needed in subsequent sections.  Before we recall some useful results, let us define the notion of solutions considered in this paper. In order to do this, let us first define Steklov average as follows: let $h \in (0,2T)$ be any positive number, then we define
\begin{equation*}
  u_{h}(\cdot,t) := \left\{ \begin{array}{ll}
                              \hint_t^{t+h} u(\cdot, \tau) \ d\tau \quad & t\in (-T,T-h), \\
                              0 & \text{else}.
                             \end{array}\right.
 \end{equation*}
We shall now define the notion of weak solutions to \cref{main_eqn}.
\begin{definition}[Weak solution]
\label{weak_sol}
    We say that $u \in C^0(-T,T;L^2_{\loc}(\Om)) \cap L^p(-T,T;W^{1,p}_{\loc}(\Om))$ is a weak solution of \cref{main_eqn} if,  for any $\phi \in C_c^{\infty}(\Om)$ and any $t \in (-T,T)$,  the following holds:
\begin{equation*}
  \int_{\Om \times \{t\}} \left\{ \frac{d [u]_{h}}{dt} \phi + \iprod{[\aa(x,t,\nabla u)]_{h}}{\nabla \phi} \right\} \,dx = 0 \txt{for any}0 < t < T-h.
\end{equation*}
\end{definition}

\begin{definition}[Function Space]\label{func_space}
For any $1<\tp<\infty$ and any $m > 1$, we define the following Banach spaces:
\begin{equation*}%
\begin{array}{c}
V^{m,\tp}(\Om_T) := L^{\infty}(-T,T;L^m(\Om)) \cap L^{\tp}(-T,T;W^{1,\tp}(\Om)),\\
V^{m,\tp}_0(\Om_T) := L^{\infty}(-T,T;L^m(\Om)) \cap L^{\tp}(-T,T;W^{1,\tp}_0(\Om)).
\end{array}\end{equation*}%
These function spaces have the norm 
\begin{equation*}
    \|f\|_{V^{m,\tp}(\Om_T)} := \sup_{-T < t<T} \|f(\cdot,t)\|_{L^m(\Om)} + \|\nabla f\|_{L^{\tp}(\Om_T)}.
\end{equation*}%
\end{definition}

We have the following parabolic Sobolev embedding theorem from \cite[Proposition 3.1 from Chapter I]{DB93}.
\begin{lemma}
\label{par_sob_emb}
For any $v \in V^{2,\tp}_0(\Om_T)$, there exists a constant $C = C(N,\tp)$ such that
\[
\iint_{\Om_T} |v(x,t)|^{\tq} \,dz \apprle C \lbr \sup_{0<t<T} \int_{\Om} |v(x,t)|^2 \,dx \rbr^{\frac{\tp}{N}} \lbr \iint_{\Om_T} |\nabla v(x,t)|^{\tp} \,dz \rbr,
\]
where $\tq:= \tp \frac{N+2}{N}$.
\end{lemma}

The next result that we need is a parabolic Sobolev embedding, see \cite[Corollary 3.1 of Chapter I]{DB93} for the details.
    \begin{lemma}\label{sobolev-poincare}
    Let $1<s<\infty$ and  $v \in V_0^{s}(Q)$ in some cylinder $Q=B \times I$, then
    \begin{equation*}
    \|v \|_{L^{s}(Q)}^{s} \leq C \abs{\{ |v| > 0\}}^{\frac{s}{N+s}} \|v \|_{V^{s}(Q)}^{s}.
    \end{equation*}
    \end{lemma}

Let us recall a form of  Poincar{\`e} inequality, see  \cite[Lemma 2.2 of Chapter I]{DB93} for the details.
\begin{lemma}\label{Poincare}
    Let $v \in W^{1,1}\left(B_{\rho}\right) \cap C^0\left(B_{\rho}\right)$ for some $\rho>0$ and  let $k, l \in \RR$ with $k<l$. Then there exists a constant $\gamma = \gamma(N,p)$,  independent of $k,l,v, \rho$, such that
\begin {equation*}
     (l-k)\left| B_{\rho}\cap \{v > l \} \right|
    \leq \gamma \frac{\rho^{N+1}}{\left| B_{\rho}\cap \{v \le k\}\right|}
    \int_{ B_{\rho}\cap \{ k < v < l\}} |\nabla v| \,dx .
\end {equation*}
    \end{lemma}

Next we recall a well known iteration lemma, see \cite[Lemma 4.1 of Chapter I]{DB93} for the details.
\begin{lemma}
    \label{iteration}
    Let $\{X_n\}$ for  $n=0,1,2,\ldots,$ be a sequence of positive numbers, satisfying the recursive inequalities
    \begin {equation*}
    X_{n+1} \leq C b^{n}  X_{n}^{1+\alpha},
    \end {equation*}
    where $C,b >1$ and $\alpha>0$ are given numbers. If
    \[
    X_0\leq C^{-\frac{1}{\alpha}}b^{-\frac{1}{\alpha^2}},
    \]
    then $\{X_n\}$ converges to zero as $n\rightarrow \infty$.
\end{lemma}

Let us recall a second important iterative estimate proved in \cite[Lemma 4.3 of Chapter I]{DB93}
\begin{lemma}\label{lemma_iter_2}
    Let $\{Y_n\}_{n=1}^{\infty}$ be a sequence of equibounded positive numbers satisfying the recursive inequalities 
    \[
        Y_n \leq C b^n Y_{n+1}^{1-\al},
    \]
    where $C,b>1$ and $\al \in (0,1)$ are given constants, then the following holds:
    \[
        Y_0 \leq  \lbr \frac{2C}{b^{1-\frac{1}{\al}}}\rbr^{\frac{1}{\al}}.
    \]

\end{lemma}

%

\subsection{Notation}
We now list some  notations that will be used throughout the paper: 
\begin{enumerate}[(i)]
    
 \item\label{not1} We shall denote a point in $\RR^{N+1}$ by $z = (x,t) \in \RR^N \times \RR$.
 
 \item We shall use the notation $Q_{a,b}(x_0,t_0)$ to denote a parabolic cylinder of the form $B_a(x_0) \times (t_0-b,t_0+b)$. 
 \item\label{not11} Henceforth, we shall fix a cylinder $Q_0 = B_{R_0} \times (-R_0^2,R_0^2)$ centered at $(0,0)$ and its scaled version $4Q_0$.
  \item\label{not12} We shall denote the boundary of $4Q_0$ by $$\Gamma  = \lbr[[]B_{4R_0} \times\left\{t=-(4R_0)^2\right\}\rbr[]] \bigcup\lbr[[] B_{4R_0} \times\left\{t=(4R_0)^2\right\} \rbr[]]\bigcup \lbr[[]\pa B_{4R_0} \times \lbr-(4R_0)^2,(4R_0)^2\rbr\rbr[]].$$
 
 \item\label{not2} Let $\rho >0$, $\la \geq 1$ and $R_0 >0$ be fixed numbers, then for a given point $z_0 = (x_0,t_0) \in \RR^{N+1}$, we define the following cylinders:
 \begin{equation*}
Q_{\rho} (x_0,t_0)  :=  B_{\rho}(x_0) \times (t_0-\rho^2, t_0+\rho^2) \txt{and}
Q_{\rho}^{\la} (x_0,t_0)  :=  B_{\la^{-1}\rho}(x_0) \times (t_0-\la^{-p}\rho^2, t_0+\la^{-p}\rho^2).
\end{equation*}

\item\label{not3} Let $\la \geq 1$ be given, then for given two points $z_1 = (x_1,t_1) \in \RR^{N+1}$ and $z_2 = (x_2,t_2) \in \RR^{N+1}$, we need the following metrics:
 \begin{equation*}
d (z_1, z_2)  := \max\{ |x_1 - x_2|,|t_1 - t_2|^{1/2}\} \txt{and} 
d (z_1, \mathcal{K}) :=  \inf_{z_2 \in \mathcal{K}} d (z_1, z_2).
\end{equation*}

\item\label{not4} Given any exponent $q \in (1,\infty)$, we shall denote $q' = \frac{q}{q-1}$ to be it's conjugate exponent.

\item $C \bigwedge 1 = \max \{ C, 1\}$.

 \item\label{not_par_bnd} For a given space-time cylinder $Q = B_R \times (a,b)$, we denote the parabolic boundary of $Q$ to be the union of the bottom and the lateral boundaries, i.e., $\pa_pQ = B_r \times \{t=a\} \bigcup \pa B_R \times (a,b)$.
 
\end{enumerate}
\section{Main Theorems}
\label{section3}
The first theorem we prove gives a rough Lipschitz bound. 
\begin{theorem}
    \label{rough_lip}
    Let $\frac{2N}{N+2} < p < \infty$ and $u$ be a local weak solution of \cref{lip_eqn} satisfying $|\nabla u| \in L^q_{\loc}$ for all $q \in [1,\infty)$, then $u$ is Lipschitz continuous. 
\end{theorem}

    The proof of \cref{lipschitz} (and \cref{cor1}, \cref{cor2}) makes use of the fact that $|\nabla u| \in L^{\infty}$ to obtain \cref{5.17} and \cref{5.18}. Thus we first prove a rough estimate in \cref{rough_lip} which gives that the gradient is bounded and then we use this fact to obtain improved and optimal quantitative estimates which are as follows.

\begin{theorem}\label{lipschitz}
    Let $\frac{2N}{N+2} < p < \infty$ and $u$ be a local weak solution of \cref{lip_eqn} satisfying $|\nabla u| \in L^{q}_{\loc}$ for all $q \in (0,\infty)$, then the following quantitative bound holds: For any $\sigma \in (0,1)$, $\ve \in (0,1)$ and parabolic cylinder $Q_{\rho,\tht} = B_{\rho} \times (-\tht,\tht)$,  there holds 
    \[
      \sup_{Q_{(\sigma\rho,\sigma\tht)}}|\nabla u| \leq  \lbr  2^{\frac{1}{\Sigma}}\mathbb{B}^{\frac{N+2}{2}}\lbr \iint_{Q_{(\rho,\tht)}} |\nabla u|^{p+\ve} \ dz\rbr^{\frac{2}{N+2}} \frac{\mathbf{C}_1 \aa}{(1-\sigma)^2}\rbr^{\frac{\mathbb{X}\Sigma}{2\ve}} \lbr 4^{{\Sigma}}\rbr^{\frac{\mathbb{X}(\mathbb{X}-2\ve)}{4\ve^2}}\bigwedge 1,
    \]
where   $\mathbb{B}, \Sigma, \aa,\mathbb{X}$ are the constants  defined in \cref{def_const_2} and $\al,\be,\ga$ to be the constants as defined in \cref{def_const_1}.
\end{theorem}
\begin{remark}
    The right hand side of \cref{lipschitz} has $\iint_{Q_{(\rho,\tht)}} v^{p+\ve} \ dz$ which holds for any $\ve \in (0,1)$. However, it is well known from the uniform higher integrability results proved  in \cite[Theorem 6.1]{adijungtae} building on the techniques  first developed in \cite{KL},  that $v = |\nabla u|  \in L^{p+\ve_0}$ for some universal $\ve_0 \in (0,1)$ with stable estimates independent of the singular or degenerate case. Let us choose $\ve = \ve_0$, then the higher integrability result gives 
    \begin{equation}\label{high_int}
      \fiint_{Q_{(\rho,\tht)}} v^{p+\ve_0} \ dz   \apprle_{(N,p,C_0,C_1)} \lbr \fiint_{Q_{(2\rho,2\tht)}} v^{p} \ dz\rbr^{1+\ve_0\de},
    \end{equation}
where $\de = \frac{1}{-\frac{N}{p} + \frac{(N+2)d}{2}}$ and $\min\left\{\frac{2}{p},1\right\}  >d >  \frac{2N}{(N+2)p}$ is some fixed exponent.
\end{remark}
\begin{remark}
    We note that \cref{high_int} can be further weakened with the use of what is called \emph{very weak solutions}. This version of higher integrability was established  in  \cite{KLVery} and a unified approach has been developed more recently in \cite{adijehan}. Thus there exists an $\ve_0$ depending only on data such that the following estimate holds
    \begin{equation*}
      \fiint_{Q_{(\rho,\tht)}} v^{p} \ dz   \apprle_{(N,p,C_0,C_1)} \lbr \fiint_{Q_{(2\rho,2\tht)}} v^{p-\ve_0} \ dz\rbr^{1+\ve_0\de},
    \end{equation*}%
where $\de = \frac{1}{-\frac{N}{p} + \frac{(N+2)d}{2}-\ve_0}$ and $\min\left\{\frac{2}{p},1\right\}  >d >  \frac{2N}{(N+2)p}$ is some fixed exponent.
\end{remark}

We now state the further interpolation estimates. The first is the analogue of \cite[Theorem 5.1' from Chapter VIII]{DB93}:
\begin{corollary}[Degenerate case]\label{cor1}
    Let $p\geq 2$ and  $u$ be a local weak solution of \cref{lip_eqn} satisfying $v = |\nabla u| \in L^q_{\loc}$ for all $q \in [1,\infty)$, then $u$ is Lipschitz continuous. Moreover,  for any $\ep \in (0,2]$, the following estimate is satisfied:
\[
     \sup_{Q_{(\sigma\rho,\sigma\tht)}}v \leq  \lbr  2^{\frac{1}{\Sigma}}\mathbb{B}^{\frac{N+2}{2}}\lbr \iint_{Q_{(\rho,\tht)}} v^{p-2+\ep} \ dz\rbr^{\frac{2}{N+2}} \frac{\mathbf{C} \aa}{(1-\sigma)^2}\rbr^{\frac{\mathbb{X}\Sigma}{2\ep}} \lbr 4^{{\Sigma}}\rbr^{\frac{\mathbb{X}(\mathbb{X}-2\ep)}{4\ep^2}}\bigwedge 1,
\]
where $\mathbb{X}$, $\Sigma$, $\aa$ and $\mathbb{B}$ are analogously computed constants with the choice $\al=\ga= 0$ and $\be = p-1$. 
\end{corollary}
Now we state the analogue of \cite[Theorem 5.2' from Chapter VIII]{DB93}:
\begin{corollary}[Singular case] \label{cor2}
    Let $\frac{2N}{N+2} < p\leq 2$ and  $u$ be a weak solution of \cref{lip_eqn} satisfying $v=|\nabla u| \in L^q_{\loc}$ for all $q \in [1,\infty)$, then $u$ is Lipschitz continuous. Moreover,  for any $\ep \in (2-p,3]$, the following estimate is satisfied:
\[
     \sup_{Q_{(\sigma\rho,\sigma\tht)}}v \leq  \lbr  2^{\frac{1}{\Sigma}}\mathbb{B}^{\frac{N+2}{2}}\lbr \iint_{Q_{(\rho,\tht)}} v^{\ep} \ dz\rbr^{\frac{2}{N+2}} \frac{\mathbf{C} \aa}{(1-\sigma)^2}\rbr^{\frac{\mathbb{X}\Sigma}{2(\ep+p-2)}} \lbr 4^{{\Sigma}}\rbr^{\frac{\mathbb{X}(\mathbb{X}-2(\ep+p-2))}{4(\ep+p-2)^2}}\bigwedge 1,
\]
where $\mathbb{X}$, $\Sigma$, $\aa$ and $\mathbb{B}$ are analogously computed constants with the choice $\al=\ga= 2-p$ and $\be = 1$. 
\end{corollary}
\begin{remark}
    Instead of choosing $k \geq 1$, if we were to equate both the terms in the expression for $\aa_n$ from \cref{aa_n}, we would get the estimate in \cref{cor1} (resp.  \cref{cor2}) with $\bigwedge 1$ replaced with $\bigwedge \lbr \frac{\rho^2}{\tht} \rbr^{\frac{1}{p-2}}$ (resp.  $\bigwedge \lbr \frac{\rho^2}{\tht} \rbr^{\frac{1}{2-p}}$)  and $\aa$  replaced by the analogous expression that comes with this calculation. This is the version that is given in \cite[Theorem 5.1 and Theorem 5.2 of Chapter VIII]{DB93}.
\end{remark}
We now state the uniform $C^{1, \alpha}$ result.

\begin{theorem}\label{holder}
    Let $\frac{2N}{N+2} <p<\infty$ and $u$ be a weak solution of \cref{main_holder} with \cref{structure_aa_holder} in force. Moreover, assume that $|\nabla u| \in L^{\infty}_{\loc}$, then given any cylinder $Q_0 = B_{R_0} \times (-R_0^2,R_0^2)$,  there exists $\al= \al(N,p,C_0,C_1) \in (0,1)$ such that for any $z_0, z_1 \in Q_0$, there holds
    \[
        |\nabla u(z_0) - \nabla u(z_1)| \leq C \mu_0^a \lbr \frac{d(z_0,z_1)}{R_0}\rbr^{\al},
    \]
    where $\mu_0 = \max\{ 1,\sup_{4Q_0} |\nabla u|\}$, $C= C(N,p,C_0,C_1)$ and $a= a(N,p,\al)$.
\end{theorem}
\begin{remark}\label{regularize}
We note that although the proof of Theorem \ref{holder} works in a uniform way for $1 < p < \infty$, but the restriction $p> \frac{2N}{N+2}$ is  only needed to regularize the equation  (as in Section 3 in  \cite{KM2})  and then getting uniform $C^{1, \alpha}$ estimates which only depends on the structure conditions. It is noted that the existence of solutions to the regularized problems in \cite{KM2} requires the use of Sobolev embedding  and that is precisely where one  requires $p > \frac{2N}{N+2}$. 
\end{remark}

\begin{remark}\label{mua} In the proof of $C^{1,\alpha}$ regularity, we shall denote the exponent '$a$' to be a number that depends on $N,p,C_0,C_1$ and the H\"older exponent $\al$. This exponent '$a$' appears on the right hand side of \cref{holder} and in the proof, by an abuse of notation, we redefine '$a$' at every step to be the larger of all the occurrences of '$a$'. This works in our case since $\mu_0 \geq 1$. 
\end{remark}

\section{Proof of rough Lipschitz bound - \texorpdfstring{\cref{rough_lip}}.}
\label{section4.5}
In order to prove uniform Lipschitz estimates, we will study the prototype equation so that we can follow some of the calculations from \cite[Chapter VIII]{DB93} for ease of reading. It must be noted that the result can be extended to more general equations of the form \cref{main_eqn} satisfying \cref{ellipticity} (see \cite[Section 1-(ii) of Chapter VIII]{DB93} for more on this) with standard modifications. Let us recall the prototype equation
\begin{equation}
    \label{lip_eqn}
    u_t - \dv( |\nabla u|^{p-2}  \nabla u) = 0.
\end{equation}

Let us first recall the well known energy estimate proved in  \cite[Proposition 3.2 of Chapter VIII]{DB93}.
\begin{lemma}
    Let $u$ be a local, weak solution of \cref{lip_eqn} and let $f(\cdot)$ be a non-negative, bounded, Lipschitz function on $\RR^+$. Then there exists a constant $C=C(N,p)$ such that on the  cylinder $Q_{(\rho,\tht)}(z_0) = B_{\rho}(x_0) \times (t_0-\tht, t_0)$, we have
    \begin{equation}\label{energy}
        \begin{array}{l}
            \left.\sup_{t_0-\tht\leq t \leq t_0} \int_{B_{\rho}(x_0)} \lbr \int_0^v s f(s) \ ds \rbr \zeta^2 \ dx \right|_{t_0-\tht}^{t} + \iint_{Q_{(\rho,\tht)}(z_0)} v^{p-2} |\nabla^2u|^2 f(v) \zeta^2 \ dz \\
            + \iint_{Q_{(\rho,\tht)}(z_0)} v^{p-1} |\nabla v|^2 f'(v) \zeta^2 \ dz + (p-2)  \iint_{Q_{(\rho,\tht)}(z_0)} v^{p-3} |\iprod{\nabla v}{\nabla u}|^2 f'(v) \zeta^2 \ dz \\
            \hfill\apprle_{C(N,p)} \iint_{Q_{(\rho,\tht)}(z_0)} v^{p} f(v) |\nabla \zeta|^2 \ dz + \iint_{Q_{(\rho,\tht)}(z_0)} \lbr \int_0^v sf(s) \ ds \rbr \zeta \zeta_t \ dz,
        \end{array}
    \end{equation}
    where we have denoted $v := |\nabla u|$.
\end{lemma}

Let us first fix some constants:
\begin{definition}\label{def_const_1}
Let $\al,\be,\ga$ be positive constants satisfying the following relations,
\begin{enumerate}[(i)]
    \item\label{item1} Choose $\ga$ such that $p-2+\ga >0$. So since we have $p>\frac{2N}{N+2}$, we let $\ga = \frac{4}{N+2}$.
    \item\label{item2} Choose $\al$ and $\be$ such that   $\al \geq \ga$ and $\be \geq 1$.
\end{enumerate}
For the proof of \cref{lipschitz},  we will  take $\al = \ga = \frac{4}{N+2}$ and $\be = p-1+\ga$ which is admissible since $p-2+\ga>0$. For the proof of \cref{cor1} ($p \geq 2$), we take $\al=\ga=0$ and $\be = p-1$ and for the proof of \cref{cor2} ($p \leq 2$), we take $\al=\ga=2-p$ and $\be = 1$. 
\end{definition}

\subsection{Energy type estimate}

Let radii $\rho, \tht$ be given and for some fixed $\sigma \in (0,1)$, let us define the following:
\begin{equation}\label{5.3}
\begin{array}{llcl}
     \tht_n:= \sigma \tht + \frac{(1-\sigma)\tht}{2^n} , \quad  & \rho_n:= \sigma \rho + \frac{(1-\sigma)\rho}{2^n} & \txt{and} &  Q_n:= Q_{(\rho_n,\tht_n)},\\
     \trho_n := \frac{\rho_n + \rho_{n+1}}{2}, \quad & \ttht_n := \frac{\tht_n + \tht_{n+1}}{2} & \txt{and}&  \tQ_n:= Q_{(\trho_n,\ttht_n)}.
     \end{array}
\end{equation}
With the above choices of radii, we note that $Q_0= Q_{(\rho,\tht)}$ and   $Q_{\infty}= Q_{(\sigma \rho,\sigma \tht)}$.
 For a fixed $k\in (0,\infty)$ to be eventually chosen, we denote 
\begin{equation}\label{k_n}
    k_n := k - \frac{k}{2^n},
\end{equation}
and consider the following cut-off functions:
\begin{equation}\label{5.5}
\begin{array}{lcl}
    \left\{ \begin{array}{l} \zeta_n = 1 \ \text{on} \ \tQ_n,\\ \zeta_n = 0 \ \text{on} \ \pa_p \tQ_n \end{array}\right. & \txt{with} &  |\nabla \zeta_n| \leq \frac{2^{n+2}}{(1-\sigma)\rho}\quad \text{and} \quad   \abs{\ddt{\zeta_n}} \leq \frac{2^{n+2}}{(1-\sigma)\tht}.
\end{array}
\end{equation}
Note that since we start off with the assumption $u \in W^{1,q}_{\loc}$ for any $q<\infty$, the choice of $f(v) = v^{\al} (v-k_{n+1})_+^{\be}$ in \cref{energy} is admissible, where  $\al,\be,\ga,\ve$ as chosen to satisfy \cref{def_const_1}, noting that $\be \geq 1$ due to \cref{item2} (see \cite[Corollary 3.1, Chapter VIII]{DB93} for more on the admissibility of the choice of $f(v)$).

Let us state the lemma that will be proved:
\begin{lemma}
    \label{energy_lemmma}
    With $\al,\be,\ga$ as in \cref{def_const_1} and $k,\tQ_n,\zeta_n$ as given above, we have the following energy estimate:
    \begin{equation}\label{energy_lemma_est}
    \begin{array}{l}
   (\alpha+ \beta+ 2 - \gamma)     \sup_{I_n} \lbr \frac{k}{2} \rbr^\ga \int_{B_n} \lbr\vkn^{\frac{\al+\be+2-\ga}{2}}\zeta_n\rbr^2 \ dx + \lbr \frac{k}{2} \rbr^{p-2+\ga}  \iint_{Q_n} \lbr \nabla \lbr\vkn^{\frac{\al+\be+2-\ga}{2}}\zeta_n\rbr\rbr^2  \ dz\\
        \hfill \apprle  (\al+\be+2)^2 \iint_{Q_n}v^{p+\al+\be} \lsb{\chi}{\{v \geq k_{n+1}\}×} |\nabla\zeta_n|^2 \ dz + ( \alpha+ \beta+2) \iint_{Q_n} v^{2+\al+\be} \lsb{\chi}{\{v \geq k_{n+1}\}×} |(\zeta_n)_t| \ dz.
    \end{array}
\end{equation}
\end{lemma}
\begin{proof}
Recalling the choice $f(v) = v^{\al} (v-k_{n+1})_+^{\be}$, we see that in order to prove the lemma, we make use of \cref{energy}, thus we estimate each of the terms from \cref{energy} as follows:
\begin{description}[leftmargin=*]
    \item[Estimate for the first term:]  Since $\zeta_n$ vanishes on the parabolic boundary of $Q_n$, the $\sup$ term appearing on the left hand side of \cref{energy} is zero, thus we can estimate this term as follows:
    \begin{equation*}
        \begin{array}{rcl}
            \sup_{t_0-\tht\leq t \leq t_0}\int_{B_n} \lbr \int_0^v s s^{\al} (s-(k_{n+1}))_+^\be \ ds \rbr \zeta_n^2 \ dx & \geq & \sup_{t_0-\tht\leq t \leq t_0} (k_{n+1})^{\ga} \int_{B_n} \lbr \int_0^v (s-(k_{n+1}))_+^{\al+ \be + 1 - \ga} \ ds \rbr \zeta_n^2 \ dx\\
            & = & \frac{(k_{n+1})^\ga}{\al+\be+2-\ga} \sup_{t_0-\tht\leq t \leq t_0}\int_{B_n} \lbr \vkn^{\frac{\al+\be+2-\ga}{2}}\zeta_n\rbr^2 \ dx.
        \end{array}
    \end{equation*}
    \textcolor{blue}{Note that here we required $1+\al \geq \ga$.}
    \item[Estimate for the second term:] We estimate the second term appearing on the left hand side of \cref{energy} as follows:
    \begin{equation*}
        \begin{array}{rcl}
            \iint_{Q_n} v^{p-2} |\nabla^ 2 u|^2 v^{\al} \vkn^{\be} \zeta_n^2 \ dz & \geq & (k_{n+1})^{p-2+\ga} \iint_{Q_n} |\nabla^ 2u|^2 \vkn^{\al+\be-\ga} \zeta_n^2 \ dz \\
            & = &\left( \frac{2}{\al+\be+2-\ga}\right)^2 (k_{n+1})^{p-2+\ga} \iint_{Q_n} \lbr \nabla\vkn^{\frac{\al+\be+2-\ga}{2}} \zeta_n\rbr^2  \ dz.
        \end{array}
    \end{equation*}
    \textcolor{blue}{Note that here we required $\al \geq \ga$.}
    \item[Estimate for the fifth term:] The first term appearing  on the right hand side of \cref{energy} is estimated as follows:
    \begin{equation*}%
        \iint_{Q_n} v^p f(v) |\nabla\zeta_n|^2 \ dz \leq \iint_{Q_n}v^{p+\al+\be} \lsb{\chi}{\{v \geq k_{n+1}\}×} |\nabla\zeta_n|^2 \ dz.
    \end{equation*}%

    \item[Estimate for the sixth term:] Analogously, the second term on the right hand side of \cref{energy} is estimated as follows:
    \begin{equation*}%
        \iint_{Q_n} \lbr \int_0^v s s^{\al} (s-k_{n+1})_+^{\be} \ ds \rbr \zeta_n (\zeta_n)_t \ dz  \leq \frac{1}{\al+\be+2}\iint_{Q_n} v^{2+\al+\be} \lsb{\chi}{\{v \geq k_{n+1}\}×} |(\zeta_n)_t| \ dz.
    \end{equation*}%
\end{description}

In \cref{energy}, in view of the fact that 
\[
 \iint_{Q_{(\rho,\tht)}(z_0)} v^{p-1} |\nabla v|^2 f'(v) \zeta^2 \ dz + (p-2)  \iint_{Q_{(\rho,\tht)}(z_0)} v^{p-3} |\iprod{\nabla v}{\nabla u}|^2 f'(v) \zeta^2 \ dz \geq 0
 \]
 for $p>1$ which follows by an application of Cauchy-Schwartz inequality, 
let us ignore the third and fourth term appearing on the left hand side and make use of the bounds obtained for the first, second, fifth and sixth terms  above to get
\begin{equation}\label{5.8}
    \begin{array}{l}
        \frac{(k_{n+1})^\ga}{\al+\be+2-\ga} \sup_{I_n} \int_{B_n} \lbr \vkn^{\frac{\al+\be+2-\ga}{2}}\zeta_n\rbr ^2 \ dx + \frac{4(k_{n+1})^{p-2+\ga}}{(\al+\be+2-\ga)^2}  \iint_{Q_n} \lbr \nabla\vkn^{\frac{\al+\be+2-\ga}{2}} \zeta_n \rbr^2 \ dz\\
        \hfill \apprle \iint_{Q_n}v^{p+\al+\be} \lsb{\chi}{\{v \geq k_{n+1}\}×} |\nabla\zeta_n|^2 \ dz + \frac{1}{\al+\be+2}\iint_{Q_n} v^{2+\al+\be} \lsb{\chi}{\{v \geq k_{n+1}\}×} |(\zeta_n)_t| \ dz.
    \end{array}
\end{equation}
From the product rule, we see that 
    \begin{equation}\label{5.9}
             \iint_{Q_n} \abs{\nabla\vkn^{\frac{\al+\be+2-\ga}{2}}}^2 \zeta_n^2 \ dz + \iint_{Q_n} \vkn^{{\al+\be+2-\ga}} |\nabla\zeta_n|^2 \ dz 
             =  \iint_{Q_n} \abs{\nabla\lbr\vkn^{\frac{\al+\be+2-\ga}{2}}\zeta_n \rbr}^2 \ dz,
    \end{equation}
    thus we estimate the second term appearing on the left hand side of \cref{5.9} as follows:
\begin{equation}\label{5.10}
    \begin{array}{rcl}
        \iint_{Q_n} \vkn^{{\al+\be+2-\ga}} |\nabla\zeta_n|^2 \ dz & = &   \iint_{Q_n} \frac{v^{{p+\al+\be}}}{v^{p-2+\ga}}\lsb{\chi}{\{v \geq k_{n+1}\}}  |\nabla\zeta_n|^2 \ dz \\
        & \leq & \frac{1}{(k_{n+1})^{p-2+\ga}} \iint_{Q_n} v^{p+\al+\be}\lsb{\chi}{\{v \geq k_{n+1}\}×} |\nabla\zeta_n|^2 \ dz.
    \end{array}
\end{equation}
Adding $\frac{4(k_{n+1})^{p-2+\ga}}{(\al+\be+2-\ga)^2}  \iint_{Q_n} \vkn^{{\al+\be+2-\ga}} |\nabla\zeta_n|^2 \ dz$ to both sides of \cref{5.8} and making use of \cref{5.9} and \cref{5.10}, we get
\begin{equation}\label{5.11}
    \begin{array}{l}
        \sup_{I_n} \frac{k_{n+1}^\ga}{\al+\be+2-\ga} \int_{B_n} \lbr \vkn^{\frac{\al+\be+2-\ga}{2}}\zeta_n\rbr^2 \ dx + \frac{4k_{n+1}^{p-2+\ga}}{(\al+\be+2-\ga)^2}  \iint_{Q_n} \lbr \nabla\lbr\vkn^{\frac{\al+\be+2-\ga}{2}}\zeta_n\rbr\rbr^2  \ dz\\
        \hfill \apprle \lbr 1 + \frac{2}{\al+\be+2-\ga}\rbr\iint_{Q_n}v^{p+\al+\be} \lsb{\chi}{\{v \geq k_{n+1}\}×} |\nabla\zeta_n|^2 \ dz + \frac{1}{\al+\be+2}\iint_{Q_n} v^{2+\al+\be} \lsb{\chi}{\{v \geq k_{n+1}\}×} |(\zeta_n)_t| \ dz.
    \end{array}
\end{equation}
Since $\al+\be+2-\ga \leq \al+\be+2$, we multiply \cref{5.11} with $(\al+\be+2-\ga)^2$ to get the desired estimate. 
\end{proof}

For any $q \in (0,\infty)$, let us recall the well known Chebyschev's  inequality:
\begin{equation}
    \label{chebyschev}
    \iint_{Q_n} \lsb{\chi}{\{v\geq k_{n+1}\}} \ dz \leq \frac{1}{k_{n}^{q}} \iint_{Q_n} \vkn[n]^{q} \ dz.
\end{equation}
\begin{remark}\label{rmk_cheb_imp}Following the calculation from \cite[Estimate (7.5) of Chapter V]{DB93}, for any $\de >1$ and some parabolic cylinder $Q$, we have
\begin{equation*}
        \iint_{Q} \vkn[n]^{\de} \ dz  \geq   \iint_{Q} \vkn[n]^{\de} \lsb{\chi}{\{v \geq k_{n+1}\}} \ dz  \geq \iint_Q v^{\de} \lbr 1- \frac{2^{n+1}-2}{2^{n+1}-1}\rbr^{\de}  \lsb{\chi}{\{v \geq k_{n+1}\}} \ dz 
         \apprge  \frac{1}{2^{n\de}} \iint_Q v^{\de}  \lsb{\chi}{\{v \geq k_{n+1}\}} \ dz.
\end{equation*}
\end{remark}

\subsection{Proof of rough Lipschitz bound}

Let us now prove that $|\nabla u| \in L^{\infty}$ with the estimate depending on the quantities $N,p,C_0,C_1$ and $\|\nabla u\|_{L^s_{\loc}}$ for some $s \in (0,\infty)$ which is finite by hypothesis. 

Let us take any $\al,\be,\ga$ such that the two conditions in \cref{def_const_1} are satisfied and define
\begin{equation*}
    Y_n:= \iint_{Q_n} \vkn[n]^{\al+ \be+2-\ga} \ dz.
\end{equation*}%
Then applying Sobolev-Poincare inequality from \cite[Proposition 3.1 of Chapter I]{DB93}, we get
\begin{equation}\label{5.15}
\begin{array}{rcl}
    Y_{n+1}  & \leq &  \lbr \iint_{Q_n} \lbr\vkn^{\frac{\al+\be+2-\ga}{2}}\zeta_n\rbr^{\frac{2(N+2)}{N}} \ dz \rbr^{\frac{N}{N+2}} \lbr \iint_{Q_n} \lsb{\chi}{\{\vkn[n+1]\geq 0\}}\rbr^{\frac{2}{N+2}}\\
    & \apprle &  \lbr \iint_{Q_n} \lbr \nabla\lbr\vkn^{\frac{\al+\be+2-\ga}{2}}\zeta_n\rbr\rbr^2 \ dz \rbr^{\frac{N}{N+2}} \lbr \sup_{I_n}\int_{B_n} \lbr\vkn^{\frac{\al+\be+2-\ga}{2}}\zeta_n\rbr^{2} \ dz \rbr^{\frac{2}{N+2}}\\
    && \times \lbr \iint_{Q_n} \lsb{\chi}{\{\vkn[n+1]\geq 0\}}\rbr^{\frac{2}{N+2}}\\
    & = & I^{\frac{N}{N+2}} \times II^{\frac{2}{N+2}} \times III^{\frac{2}{N+2}}.
\end{array}
\end{equation}
We shall estimate each of the terms appearing on the right hand side of \cref{5.15} as follows:
\begin{description}[leftmargin=*]
    \item[Estimate for $I$:] Making use of \cref{energy_lemma_est}, we get
    \begin{equation}
        \label{estimate_I_p}
            I  \apprle  \frac{1}{k^{p-2+\ga}}\lbr \iint_{Q_n}v^{p+\al+\be} \lsb{\chi}{\{v \geq k_{n+1}\}×} |\nabla\zeta_n|^2 \ dz + \iint_{Q_n} v^{2+\al+\be} \lsb{\chi}{\{v \geq k_{n+1}\}×} |(\zeta_n)_t| \ dz\rbr
    \end{equation}
    Let us estimate each of these terms as follows:
    \begin{equation}\label{5.17_p}
        \begin{array}{l}
            \iint_{Q_n}v^{p+\al+\be} \lsb{\chi}{\{v \geq k_{n+1}\}×} |\nabla\zeta_n|^2 \ dz  \leq  \frac{2^{2(n+2)}}{\rho^2}\iint_{Q_n}v^{A}v^{p+\al+\be-A} \lsb{\chi}{\{v \geq k_{n+1}\}×}\\
             \qquad \qquad \qquad \leq  \frac{2^{2(n+2)}}{\rho^2} \lbr \iint_{Q_n}v^{\al+\be+2-\ga} \lsb{\chi}{\{v \geq k_{n+1}\}} \ dz \rbr^{\frac{A}{\al+\be+2-\ga}} \lbr \iint_{Q_n}v^{\frac{(p+\al+\be)(\al+\be+2-\ga)}{\al+\be+2-\ga-A}} \rbr^{\frac{\al+\be+2-\ga-A}{\al+\be+2-\ga}},
        \end{array}
    \end{equation}
    where  $A = \frac{\al+\be+2-\ga}{2}$.
    \begin{equation}\label{5.18_p}
        \begin{array}{l}
            \iint_{Q_n}v^{2+\al+\be} \lsb{\chi}{\{v \geq k_{n+1}\}×} |(\zeta_n)_t| \ dz \leq  \frac{2^{(n+2)}}{\theta}\iint_{Q_n}v^{B}v^{2+\al+\be-B} \lsb{\chi}{\{v \geq k_{n+1}\}×}\\
           \qquad \qquad \qquad  \leq \frac{2^{(n+2)}}{\theta} \lbr \iint_{Q_n}v^{\al+\be+2-\ga} \lsb{\chi}{\{v \geq k_{n+1}\}} \ dz \rbr^{\frac{B}{\al+\be+2-\ga}} \lbr \iint_{Q_n}v^{\frac{(2+\al+\be)(\al+\be+2-\ga)}{\al+\be+2-\ga-B}} \rbr^{\frac{\al+\be+2-\ga-B}{\al+\be+2-\ga}},
        \end{array}
    \end{equation}
    where $B = \frac{\al+\be+2-\ga}{2}$.
Thus combining \cref{5.17_p} and \cref{5.18_p} into \cref{estimate_I_p} gives
\begin{equation}
    \label{estimate_I}
    I \apprle \frac{2^{2(n+2)} }{k^{p-2+\ga}}\lbr \frac{\|v\|_{2(p+\al+\be)}^{p+\al+\be}}{\rho^2} Y_n^{\frac12}  + \frac{\|v\|_{2(2+\al+\be)}^{2+\al+\be}}{\tht} Y_n^{\frac12}\rbr.
\end{equation}
    \item[Estimate for $II$:]
    Making use of \cref{energy_lemma_est} and proceeding analogous to \cref{estimate_I}, we get
    \begin{equation}
        \label{estimate_II}
            II  \apprle  \frac{2^{2(n+2)} }{k^{\ga}}\lbr \frac{\|v\|_{2(p+\al+\be)}^{p+\al+\be}}{\rho^2} Y_n^{\frac12}  + \frac{\|v\|_{2(2+\al+\be)}^{2+\al+\be}}{\tht} Y_n^{\frac12}\rbr.
    \end{equation}
    \item[Estimate for $III$:] We can apply \cref{chebyschev} to get
    \begin{equation}\label{estimate_III}
        III \apprle \frac{2^{n(\al+\be+2-\ga)}}{k^{\al+\be+2-\ga}} Y_n.
    \end{equation}
\end{description}
Let us denote a constant $\mathbf{B}_1$ to possibly depend on $N,p,C_0,C_1, \rho, \tht, \|v\|_{2(2+\al+\be)}, \|v\|_{2(2+\al+\be)},\|v\|_{2(p+\al+\be)},\al,\be,\ga$ noting that each of these quantities are finite due to the hypothesis. Thus making use of \cref{estimate_I}, \cref{estimate_II}, \cref{estimate_III} into \cref{5.15}, we get
\begin{equation*}%
Y_{n+1} \leq \frac{\mathbf{B}_1 }{k^{\mathbb{E}}} Y_n^{1 + \frac{1}{N+2}} \mathbf{D}^n,
\end{equation*}%
where $\mathbb{E} := \frac{(p-2+\ga)N + 2(\al+\be+2)}{N+2}$ and $\mathbf{D} = 2^{2 + \frac{2(\al+\be+2-\ga)}{N+2}}$.  If we make the choice of $k$ such that
\begin{equation*}%
    Y_0  \leq  \lbr \frac{\mathbf{B}_1}{k^{\mathbb{E}}} \rbr^{-(N+2)} \mathbf{D}^{-(N+2)^2},
\end{equation*}%
then from \cref{iteration},  we see that $\{Y_n\}$ converges to zero as $n \rightarrow \infty$. This concludes the proof of \cref{rough_lip}. Before proceeding further, we make the following discursive remark.

\begin{remark}
    In the strictly degenerate regime $p \geq 2$ or the singular regime $p \leq 2$, such an explicit Lipschitz estimate is given in \cite[Chapter VIII]{DB93}. In the proof of the uniform rough Lipschitz estimate  as in \cref{rough_lip} above,  we however closely follow the approach  in \cite{DBF} ( where this Lipschitz regularity result was first proved)  instead of that in \cite{DB93} because we couldn't find  a direct  adaptation of  the proof of Lemma 4.2 in Chapter VIII  of \cite{DB93} to our situation.  \end{remark}

\section{Proof of Uniform Lipschitz estimate - \texorpdfstring{\cref{lipschitz}}.}
\label{section5}

Thanks to \cref{rough_lip}, we see that $|\nabla u| \in L^{\infty}_{\loc}$ and thus we can prove the required quantitative estimates. The first step is to iterate the energy estimates by making careful choice of the level sets. 

\subsection{First Iteration}
Let us now take $\al=\ga$ and $\be = p-1+\ga >1$ and then prove the quantitative Lipschitz regularity.  The proof of the Lipschitz regularity requires two iterative steps, the first of which is proved in this subsection. With $k_n$ as defined in \cref{k_n},  let us set
\begin{equation*}
    Y_n:= \iint_{Q_n} \vkn[n]^{p+1+\ga} \ dz.
\end{equation*}%
Since we want to suitably control $Y_{n+1}$ in terms of $Y_n$, we proceed as follows:
\begin{equation}\label{1.19}
    \begin{array}{rcl}
        Y_{n+1} & \overset{\redlabel{514a}{a}}{\leq} & \iint_{Q_n} \lbr\vkn^{\frac{p+1+\ga}{2}}\zeta_n\rbr^2 \ dz \\
        & \overset{\redlabel{514b}{b}}{\leq}& \lbr  \iint_{Q_n} \lbr\vkn^{\frac{p+1+\ga}{2}}\zeta_n\rbr^{\frac{2(N+2)}{N}} \ dz\rbr^{\frac{N}{N+2}} \lbr \iint_{Q_n} \lsb{\chi}{\{v \geq k_{n+1}\}} \ dz \rbr^{\frac{2}{N+2}}\\
        & \overset{\redlabel{514c}{c}}{\apprle} & \lbr  \iint_{Q_n} \lbr \nabla\lbr \vkn^{\frac{p+1+\ga}{2}}\zeta_n\rbr\rbr^{2} \ dz\rbr^{\frac{N}{N+2}} \lbr \sup_{I_n} \int_{B_n}  \lbr \vkn^{\frac{\al+\be+2-\ga}{2}}\zeta_n\rbr^2 \ dz \rbr^{\frac{2}{N+2}}\\
        && \times\lbr \frac{Y_n}{k_{n+1}^{p+1+\ga}} \rbr^{\frac{2}{N+2}},
    \end{array}
\end{equation}
where to obtain \redref{514a}{a}, we enlarged the domain of integration and make use of \cref{5.3} and \cref{5.5}, to obtain \redref{514b}{b}, we applied H\"older's inequality and finally to obtain \redref{514c}{c}, we made use of Sobolev embedding theorem along with \cref{chebyschev}. Over here, we also note that with our choice of $\alpha, \beta$ and $\gamma$, we have that $\alpha+ \beta+ 2 - \gamma= p+1+ \gamma$.

We can estimate the first two terms appearing on the right hand side of \cref{1.19} by making use of \cref{energy_lemmma}. Let us thus estimate each of the terms appearing on the right hand side of \cref{energy_lemma_est} as follows: 
\begin{description}[leftmargin=*]
    \item[Estimate of first term:]  We estimate the first term appearing on the right hand side of \cref{energy_lemma_est} as follows:
    \begin{equation}\label{5.17}
        \begin{array}{rcl}
            \iint_{Q_n}v^{p+\al+\be} \lsb{\chi}{\{v \geq k_{n+1}\}×} |\nabla\zeta_n|^2 \ dz & = & \iint_{Q_n}v^{p+\ga + p+1+\ga -2} \lsb{\chi}{\{v \geq k_{n+1}\}×} |\nabla\zeta_n|^2 \ dz\\
            & \overset{\cref{5.5}}{\leq} & \frac{4^{n+2}}{(1-\sigma)^2\rho^2} \frac{\lbr\sup_{Q_n} v\rbr^{p + \ga}}{(k_{n+1})^{2}} \iint_{Q_n} v^{p+1+\ga} \lsb{\chi}{\{v \geq k_{n+1}\}} \ dz\\
            & \overset{\text{\cref{rmk_cheb_imp}}}{\leq} & \frac{4^{n+2}2^{n(p+1+\ga)}}{(1-\sigma)^2\rho^2} \frac{\lbr\sup_{Q_n} v\rbr^{p + \ga}}{(k_{n+1})^{2}} Y_n.
        \end{array}
    \end{equation}
    \item[Estimate of second term:] Analogously, we estimate the second term appearing on the right hand side of \cref{energy_lemma_est} as follows:
    \begin{equation}\label{5.18}
        \begin{array}{rcl}
            \iint_{Q_n}v^{2+\al+\be} \lsb{\chi}{\{v \geq k_{n+1}\}×} (\zeta_n)_t \ dz & = & \iint_{Q_n}v^{2+\ga+p-1+\ga+p-p} \lsb{\chi}{\{v \geq k_{n+1}\}×} (\zeta_n)_t \ dz\\
            & \overset{\cref{5.5}}{\leq} & \frac{2^{n+2}}{(1-\sigma)\tht} \frac{\lbr\sup_{Q_n} v\rbr^{p + \ga}}{(k_{n+1})^{p}} \iint_{Q_n} v^{p+1+\ga} \lsb{\chi}{\{v \geq k_{n+1}\}} \ dz\\
            & \overset{\text{\cref{rmk_cheb_imp}}}{\leq} & \frac{2^{n+2}2^{n(p+1+\ga)}}{(1-\sigma)^2\tht} \frac{\lbr\sup_{Q_n} v\rbr^{p + \ga}}{(k_{n+1})^{p}}Y_n.
        \end{array}
    \end{equation}
\end{description}

Making use of \cref{5.17} and \cref{5.18} into \cref{energy_lemma_est} gives the following two bounds:
\begin{equation}\label{energy_lemma_est_no_k}
    \begin{array}{l}
        \sup_{I_n}   \int_{B_n} \lbr \vkn^{\frac{p+1+\ga}{2}}\zeta_n\rbr^2 \ dx \\
\hspace*{2cm} \apprle \frac{4^{n+2}2^{n(p+1+\ga)}}{(1-\sigma)^2}\lbr\sup_{Q_n} v\rbr^{p+ \ga} \frac{1}{(k_{n+1})^{2}}\lbr \frac{1}{\rho^2} \frac{1}{(k_{n+1})^{\ga}}  + \frac{1}{\tht}\frac{1}{(k_{n+1})^{p+\ga-2}} \rbr Y_n.
    \end{array}
    \end{equation}
    \begin{equation}\label{energy_lemma_2_est_no_k}
    \begin{array}{l}
           \iint_{Q_n} \lbr \nabla\lbr\vkn^{\frac{p+1+\ga}{2}}\zeta_n\rbr\rbr^2  \ dz\\
\hspace*{2cm} \apprle \frac{4^{n+2}2^{n(p+1+\ga)}}{(1-\sigma)^2}\lbr\sup_{Q_n} v\rbr^{p + \ga} \frac{1}{(k_{n+1})^{p}}\lbr \frac{1}{\rho^2} \frac{1}{(k_{n+1})^{\ga}}  + \frac{1}{\tht}\frac{1}{(k_{n+1})^{p+\ga-2}} \rbr Y_n.
    \end{array}
    \end{equation}

    Before we proceed, let us define a few constants and  make an important remark:
    \begin{definition}\label{def_const_2}
    Let us define the following exponents required for the proof:
    \begin{multicols}{2}
    \begin{enumerate}[(i)]
    \item $\mathbb{B}:= 2^{p+2+\ga - \frac{2(p+1+\ga)}{N+2}}$
    \item $\Sigma:=\frac{N+2}{pN+4+2(p+1+\ga)}$ 
    \item $\aa:= \lbr \frac{2^{\ga}}{\rho^2}   + \frac{2^{p+\ga-2}}{\tht} \rbr$.
    \item $\mathbb{X}:= pN+4+2(p+1+\ga)$.
\end{enumerate}
\end{multicols}
\end{definition}

    \begin{remark}\label{rmk5.6}
        Since $1 \geq \frac{2^n-1}{2^n}\geq \frac12$ for any $n \geq 1$, we have 
        \begin{equation}\label{aa_n}\aa_n := \lbr \frac{1}{\rho^2} \frac{1}{(k_{n+1})^{\ga}}  + \frac{1}{\tht}\frac{1}{(k_{n+1})^{p+\ga-2}} \rbr \leq \lbr \frac{1}{\rho^2} 2^{\ga}  + \frac{1}{\tht}2^{p+\ga-2} \rbr = \aa,\end{equation} where $\aa$ is as defined in \cref{def_const_2} and $k_{n+1}$ is as defined in \cref{k_n} for some   $k \geq 1$. Here we also use the fact that $p+\ga-2 > 0$. 
        
        It is easy to see that if we were to  balance both the terms in $\aa_n$, then we go back to the original estimate arising in the proofs of   \cite[Theorems 5.1 and 5.2, Chapter VIII]{DB93}, which is where the distinction between the singular and degenerate cases was needed. 
    \end{remark}

Substituting \cref{energy_lemma_est_no_k} and \cref{energy_lemma_2_est_no_k} along with \cref{aa_n} into \cref{1.19} and recalling the definition of constants from \cref{def_const_2}, we have the following iterative estimate
\begin{equation*}%
        Y_{n+1} 
         \leq  \mathbf{C}_1  \frac{\mathbb{B}^{n}}{k^{\frac{pN+4+2(p+1+\ga)}{N+2}}} \frac{1}{(1-\sigma)^2} \lbr \sup_{Q_{(\rho,\tht)}} v\rbr^{p+\ga} Y_n^{\frac{2}{N+2}+1} \aa.
\end{equation*}%
Applying \cref{iteration}, we see that $Y_n \rightarrow 0$ as $n \rightarrow \infty$ if 
\begin{equation*}%
    Y_0 = \iint_{Q_{(\rho,\tht)}} v^{p+1+\ga} \ dz \leq  \lbr \frac{\mathbf{C}_1 \aa }{k^{\frac{pN+4+2(p+1+\ga)}{N+2}}(1-\sigma)^2} \lbr \sup_{Q_{(\rho,\tht)}} v\rbr^{p+\ga} \rbr^{-\frac{N+2}{2}} \mathbb{B}^{-\frac{(N+2)^2}{4}}.
\end{equation*}%
In particular, if we make the following choice of $k$
\begin{equation*}%
    k =  \lbr \mathbb{B}^{\frac{N+2}{2}}\lbr \iint_{Q_{(\rho,\tht)}} v^{p+1+\ga} \ dz\rbr^{\frac{2}{N+2}} \frac{\mathbf{C}_1 \aa}{(1-\sigma)^2} \lbr \sup_{Q_{(\rho,\tht)}} v\rbr^{p+\ga} \rbr^{\frac{N+2}{pN+4+2(p+1+\ga)}} \bigwedge 1,
\end{equation*}%
then $Y_n \rightarrow 0$ as $n \rightarrow \infty$. From this, we see that $Y_{\infty} = 0$ is equivalent to the following estimate:
\begin{equation}\label{sup_bnd_first}
    \sup_{Q_{(\sigma\rho,\sigma\tht)}} v \leq k =  \lbr \mathbb{B}^{\frac{N+2}{2}}\lbr \iint_{Q_{(\rho,\tht)}} v^{p+1+\ga} \ dz\rbr^{\frac{2}{N+2}} \frac{\mathbf{C}_1 \aa}{(1-\sigma)^2} \lbr \sup_{Q_{(\rho,\tht)}} v\rbr^{p+\ga} \rbr^{\frac{N+2}{pN+4+2(p+1+\ga)}} \bigwedge 1,
\end{equation}

\subsection{Second iteration}
The first iteration gives a bound of the form \cref{sup_bnd_first} which does not give the desired Lipschitz bound due the $\sup$ term appearing also on the right hand side. In order to overcome this, we now iterate \cref{sup_bnd_first} a second time in this subsection.

Let us fix some $\ve \in(0,1)$ (to be eventually chosen later) and  rewrite \cref{sup_bnd_first} as follows:
\begin{equation}\label{5.25}
    \sup_{Q_{(\sigma\rho,\sigma\tht)}} v \leq  \lbr \mathbb{B}^{\frac{N+2}{2}}\lbr \iint_{Q_{(\rho,\tht)}} v^{p+\ve} \ dz\rbr^{\frac{2}{N+2}} \frac{\mathbf{C}_1 \aa}{(1-\sigma)^2} \lbr \sup_{Q_{(\rho,\tht)}} v\rbr^{p+\ga + \frac{2(1+\ga-\ve)}{N+2}} \rbr^{{\Sigma}} \bigwedge 1.
\end{equation}
With $\sigma \in (0,1)$, consider the family of cylinders $Q_n:= Q_{(\rho_n,\tht_n)}$ where
\begin{equation*}%
    \rho_n:= \sigma \rho + (1-\sigma)\rho \sum_{i=1}^n 2^{-i} \txt{and} \tht_n:= \sigma \tht + (1-\sigma)\tht \sum_{i=1}^n 2^{-i}.
\end{equation*}%
Then we have    $Q_0 = Q_{(\sigma\rho,\sigma\tht)}$ and $Q_{\infty} = Q_{(\rho,\tht)}$. If we denote
\begin{equation*}
    M_n:= \sup_{Q_n} v,
\end{equation*}%
then \cref{5.25} applied over $Q_{n+1}$ and $Q_n$ can be rewritten  for any $\ve\in (0,1]$ as 
\begin{equation}\label{5.28}
M_{n} \leq 2^{{2n\Sigma}} M_{n+1}^{\frac{\mathbb{X} -2\ve}{\mathbb{X}}} \lbr  \mathbb{B}^{\frac{N+2}{2}} \lbr \iint_{Q_{(\rho,\tht)}} v^{p+\ve} \ dz\rbr^{\frac{2}{N+2}} \frac{\mathbf{C}_1 \aa}{(1-\sigma)^2}\rbr^{{\Sigma}} \bigwedge 1
\end{equation}
where $\mathbb{X} = (p+\ga)(N+2)+2(1+\ga) = pn+4 + 2(p+1+\ga)$ ( given our choice of $\ga$) and   thus we can rewrite \cref{5.28} as 
\begin{equation*}%
M_{n} \leq 2^{{2n\Sigma}} M_{n+1}^{1-\frac{2\ve}{\mathbb{X}}} \lbr  \mathbb{B}^{\frac{N+2}{2}} \lbr \iint_{Q_{(\rho,\tht)}} v^{p+\ve} \ dz\rbr^{\frac{2}{N+2}} \frac{\mathbf{C}_1 \aa}{(1-\sigma)^2}\rbr^{{\Sigma}} \bigwedge 1
\end{equation*}%
Thus iterating the above estimate using \cref{lemma_iter_2}, we get
\begin{equation*}%
    M_0 = \sup_{Q_{(\sigma\rho,\sigma\tht)}}v \leq  \lbr  2^{\frac{1}{\Sigma}}\mathbb{B}^{\frac{N+2}{2}}\lbr \iint_{Q_{(\rho,\tht)}} v^{p+\ve} \ dz\rbr^{\frac{2}{N+2}} \frac{\mathbf{C}_1 \aa}{(1-\sigma)^2}\rbr^{\frac{\mathbb{X}\Sigma}{2\ve}} \lbr 4^{{\Sigma}}\rbr^{\frac{\mathbb{X}(\mathbb{X}-2\ve)}{4\ve^2}}\bigwedge 1,
\end{equation*}%
which gives the desired estimate.

\section{Proof of \texorpdfstring{\cref{cor1}}.}

In \cref{5.11}, let us take $\al = \ga = 0$ and $\be = p-1$ which are admissible since $p \geq 2$. Then the analogue of \cref{sup_bnd_first} in the case $p \geq 2$ becomes
\begin{equation*}
        \sup_{Q_{(\sigma\rho,\sigma\tht)}} v \leq k =  \lbr \mathbb{B}^{\frac{N+2}{2}}\lbr \iint_{Q_{(\rho,\tht)}} v^{p+1} \ dz\rbr^{\frac{2}{N+2}} \frac{\mathbf{C}_1 \aa}{(1-\sigma)^2} \lbr \sup_{Q_{(\rho,\tht)}} v\rbr^{p-2} \rbr^{\frac{N+2}{pN-2N+2(p+1)}} \bigwedge 1,
\end{equation*}%
Note that the constants $\mathbb{B}$, $\mathbf{C}_1$ and $\aa$ are the analogous versions of those  defined in \cref{def_const_2} with this specific choices of  $\al,\be,\ga$, but by an abuse of notation, we still use the same symbols.

For any $\ep \in (0,2]$, we get the following analogue of \cref{5.28}
\begin{equation*}
M_{n} \leq 2^{{2n\Sigma}} M_{n+1}^{\frac{\mathbb{X} -2\ep}{\mathbb{X}}} \lbr  \mathbb{B}^{\frac{N+2}{2}} \lbr \iint_{Q_{(\rho,\tht)}} v^{p-2+\ep} \ dz\rbr^{\frac{2}{N+2}} \frac{\mathbf{C}_1 \aa}{(1-\sigma)^2}\rbr^{\frac{N+2}{pN+4+2(p+1)}} \bigwedge 1,
\end{equation*}%
where $\mathbb{X} ={N(p-2)+2(p+1)}$ and $\Sigma = {\frac{N+2}{N(p-2)+2(p+1)}}$. Thus iterating the above estimate using \cref{lemma_iter_2} gives the desired estimate.

\section{Proof of \texorpdfstring{\cref{cor2}}.}
In \cref{5.11}, let us take $\al = \ga = 2-p$ and $\be = p-1+\ga = 1$ which are admissible since $\frac{2N}{N+2} < p \leq 2$. Then the analogue of \cref{sup_bnd_first} in this case  becomes
\begin{equation*}
        \sup_{Q_{(\sigma\rho,\sigma\tht)}} v \leq k =  \lbr \mathbb{B}^{\frac{N+2}{2}}\lbr \iint_{Q_{(\rho,\tht)}} v^{3} \ dz\rbr^{\frac{2}{N+2}} \frac{\mathbf{C}_1 \aa}{(1-\sigma)^2} \lbr \sup_{Q_{(\rho,\tht)}} v\rbr^{2-p} \rbr^{\frac{N+2}{(2-p)N+6}} \bigwedge 1,
\end{equation*}%
Note that the constants $\mathbb{B}$, $\mathbf{C}_1$ and $\aa$ are the analogous versions of those  defined in \cref{def_const_2} with this specific choices of  $\al,\be,\ga$, but by an abuse of notation, we still use the same symbols.

For any $\ep \in (2-p,3]$, we get the following analogue of \cref{5.28}
\begin{equation*}
M_{n} \leq 2^{{2n\Sigma}} M_{n+1}^{\frac{\mathbb{X} -2(\ep+p-2)}{\mathbb{X}}} \lbr  \mathbb{B}^{\frac{N+2}{2}} \lbr \iint_{Q_{(\rho,\tht)}} v^{\ep} \ dz\rbr^{\frac{2}{N+2}} \frac{\mathbf{C}_1 \aa}{(1-\sigma)^2}\rbr^{\frac{N+2}{(2-p)N+6}} \bigwedge 1,
\end{equation*}%
where $\mathbb{X} :={6+p(N+2)}$ and $\Sigma = {\frac{N+2}{(2-p)N+6}}$ which are defined in \cref{def_const_2} with $\al=\ga =2-p$ and $\be = 1$. We see that $\ep+p-2 >0 \Longleftrightarrow 2-p < \ep$, thus iterating the above estimate using \cref{lemma_iter_2} gives the desired estimate.

\section{Proof of Uniform \texorpdfstring{$C^{1,\alpha}$}. estimate - \texorpdfstring{\cref{holder}}.}
\label{section6}
As previously mentioned in Remark \ref{regularize}, the proof of Theorem \ref{holder} is via   regularization scheme  as in \cite{KM2}. With no further mention, this aspect will be implicit in our proofs.   With $Q_0 = B_{R_0} \times (-R_0^2,R_0^2)$, let us consider equations of the form 
\begin{equation}
    \label{main_holder}
    u_t - \dv \aa(\nabla u) = 0 \txt{on} 4Q_0,
\end{equation}
with $\aa(\zeta)$ satisfying the following structural assumptions for some $s \in (0,1]$:
\begin{equation}
    \label{structure_aa_holder}
        \begin{array}{c}
        |\aa(z)| + | \aa'(z)| (|z|^2 + s^2)^{\frac12} \leq C_1 (|z|^2 + s^2)^{\frac{p-1}{2}}\\
        \iprod{\aa'(z) \zeta}{\zeta} \geq C_0 (|z|^2 + s^2)^{\frac{p-2}{2}}|\zeta|^2,
    \end{array}
\end{equation}
where we have denoted $\aa'(z) := \frac{d\aa(z)}{dz}$. Let us fix the following constant:
\begin{equation}\label{sup_u}
    \mu_0:= \max\{1, \sup_{4Q_0} |\nabla u|\},
\end{equation}
then for any $z \in Q_0$, we consider the cylinder $Q_S^{\mu_0}(z) =B_{\mu_0^{-1} S}(x) \times (t - \mu_0^{-p}S^2, t + \mu_0^{-p} S^2)$ to be the largest cylinder such that $Q_S^{\mu_0}(z) \subset 4Q_0$ and $Q_S^{\mu_0}(z) \cap (4Q_0 \setminus 2Q_0) \neq \emptyset$. Note that this fixes the radius $S$ and is independent of  the point $z \in Q_0$.  As a consequence, we have the following observations:
\begin{description}
        \descitem{O1}{obs1} Since $Q_S^{\mu_0}(z) \subset 4Q_0$, we see that  $S \leq \min\{\mu_0,\mu_0^{p/2}\}3R_0$ must hold. 
    \descitem{O2}{obs2} Moreover, since $Q_S^{\mu_0}(z)$ has to go outside $2Q_0$, we note that $S \geq  \min\{ \mu_0, \mu_0^{p/2}\} R_0$.
\end{description}
\begin{center}
\begin{tikzpicture}[remember picture,scale=0.5,>=latex]
\coordinate  (O) at (0,0);
\draw[thick, draw=blue] (-4,-2) rectangle (0,2);
\draw[thick, draw=blue] (-7,-5) rectangle (3,5);
\draw[draw=black,draw=none, fill=black, opacity=0.1] (-3,-1.5) rectangle (-2,5);
\draw[draw=black] (-3,-1.5) rectangle (-2,5);
\draw[black, fill] (-2.5,1.75) circle [radius =2pt];
\draw[black, fill] (-7,1.75) circle [radius =2pt];
\draw[black, fill] (-2.5,5) circle [radius =2pt];
\draw[draw=orange, dashed, <->] (3,1.75) -- (0,1.75);
\draw[draw=orange, dashed, <->] (-1,5) -- (-1,2);
\node  at (1.5,1.45) {\tiny $3R_0$};
\node  at (-0.4,3.5) {\tiny $15R_0^2$};

\draw[draw=teal, dashed, <->] (-4,-2.5) -- (0,-2.5);
\draw[draw=teal, dashed, <->] (-7,-5.5) -- (3,-5.5);
\draw[draw=teal, dashed, <->] (-3,-1) -- (-2,-1);
\draw[draw=teal, dashed, <->] (-4.5,-2) -- (-4.5,2);
\draw[draw=teal, dashed, <->] (-7.5,-5) -- (-7.5,5);
\draw[draw=teal, dashed, <->] (-1.5,-1.5) -- (-1.5,5);
\node  at (-1.3,-3) {\scriptsize $2R_0$};
\node  at (-1.3,-6) {\scriptsize $8R_0$};

\node  at (-5.4,0) {\scriptsize $2(R_0)^2$};
\node  at (-8.6,0) {\scriptsize $2(4R_0)^2$};

\node  at (-2.5,-0.7) {\tiny $2\mu_0^{-1} S$};
\node  at (-0.6,1) {\tiny $2\mu_0^{-p}S^2$};
\node  at (-2.5,1.5) {\tiny $(x,t)$};
\node  at (-6.4,1.4) {\tiny $(\bar{x},t)$};
\node  at (-2.5,5.3) {\tiny $(x,\bar{t})$};

\node  at (0.7,0) {$Q_0$};
\node  at (3.7,0) {$4Q_0$};
\node  at (-4,4) {\scriptsize $Q_S^{\mu_0}(z)$};
\end{tikzpicture}
\end{center}


 Since the proof will be independent of the point $z \in Q_0$, we shall ignore writing the location of the cylinder $Q_R^{\mu}$. The proof of gradient H\"older regularity requires two propositions which are given below.
\begin{proposition}\label{alt1}
    For some $\mu \leq \mu_0$ and $R \leq S$, there exists numbers $\nu \in (0,1/2)$ and $\kappa,\de \in (0,1)$ depending only on $(N,p,C_0,C_1)$ such that if 
    \begin{equation*}%
        \abs{\{z \in Q_R^{\mu} : |\nabla u(z)| < \mu/2\}}< \nu |Q_R^{\mu}| \txt{and}   \sup_{Q_R^{\mu}} |\nabla u| \leq \mu,\quad s\leq \mu,
    \end{equation*}%
    is satisfied, then the following conclusion follows:
    \begin{equation*}%
        \iint_{Q_{\de^{i+1}R}^{\mu}} |\nabla u - \avgs{\nabla u}{Q_{\de^{i+1}R}^{\mu}}|^2 \ dz \leq \ka \de^{N+2} \iint_{Q_{\de^{i}R}^{\mu}} |\nabla u - \avgs{\nabla u}{Q_{\de^{i}R}^{\mu}}|^2 \ dz
    \end{equation*}%
    for all $i \in \mathbb{Z}$.
\end{proposition}
\begin{proposition}\label{alt2}
   For some $\mu \leq \mu_0$ and $R \leq S$ with $\nu$ as fixed in \cref{alt1}, there exists numbers $\sigma , \eta \in (0,1)$ such that if 
    \begin{equation*}%
        \abs{\{z \in Q_R^{\mu} : |\nabla u(z)| < \mu/2\}}\geq \nu |Q_R^{\mu}| \txt{and} s \leq \mu,
    \end{equation*}%
    is satisfied, then the following conclusion follows:
    \begin{equation*}%
        |\nabla u(z)| \leq \eta \mu \qquad \text{for all}\  z \in Q_{\sigma R}^{\mu}.
    \end{equation*}%
\end{proposition}
\begin{remark}\label{remark8.3}
    Let us explain a little more about the condition $s \leq \mu$ in \cref{alt2}. Since we apply \cref{alt2} starting from $Q_{S}^{\mu_0}$, we automatically have $\sup_{Q_S^{\mu_0}} |\nabla u| \leq \mu_0$ and since $ s \in (0,1) \Longrightarrow s \leq \mu_0$,  we get $$s + \sup_{Q_S^{\mu_0}} |\nabla u| \leq 2\mu_0,$$ which is the hypothesis needed in \cref{prop3.1}. From this, we get that $\sup_{Q_{R_1}^{\mu_1}} |\nabla u| \leq \mu_1$, but there is no bound obtainable for $s$ in terms of $\mu_1$. Thus we need to assume $s \leq \mu_1$ in order to satisfy the hypothesis from \cref{prop3.1} to be able to iterate \cref{alt2}. We proceed this way until either the measure density hypothesis fails or $s \geq \mu_i$. 
    
    If the measure density condition fails first at step $i_0$ and  $s \leq \mu_{i_0}$, then the conclusion of \cref{alt2} at step $i_0$ gives $\sup_{Q_{R_{i_0}}^{\mu_{i_0}}} |\nabla u|\leq \mu_{i_0}$, which satisfies the hypothesis needed to apply \cref{alt1}. 
    
    On the other hand if at step $i_0$, if  $s > \mu_{i_0}$, then we are automatically in the uniformly parabolic situation  after the rescaling as in \eqref{w_rescale} and we can directly apply \cref{lemma3.2} to get the conclusion of \cref{alt1}. This is possible since we have 
    \[
        \mu_{i_0} \leq s \leq s + \sup_{Q_{R_{i_0}}^{\mu_{i_0}}} |\nabla u|\leq \mu_{i_0-1} + \mu_{i_0} = \lbr \frac{1+\eta}{\eta}\rbr \mu_{i_0},
    \]
    which immediately gives the uniform ellipticity condition in  \cref{w_x_i_sol}. Recall that $\eta$ is a universal constant. 
\end{remark}

\subsection{Covering argument}
  Let us first fix some notation needed for the proof:
\begin{definition}\label{iter_const_def}
   Let $\mu_0$ be as in \cref{sup_u} and $\eta$, $\sigma$ be as in \cref{alt2}. Then let us denote
    \begin{equation*}
        \mu_{n+1} := \eta \mu_n,   \qquad R_0 := S, \qquad R_{n+1} := c_0  R_n = c_0^{n+1} S,
    \end{equation*}
    where $c_0 := \frac12\sigma \min\{{\eta}, {\eta}^{p/2}\}$. Here we note that $c_0 \in (0,1)$ since $\sigma,\eta \in (0,1)$ and $\mu \geq 1$.
\end{definition}
We have the following result regarding the intrinsic geometry.
\begin{claim}\label{cyl_incl}
    With the notation as in \cref{iter_const_def}, we have the inclusion $Q_{c_0 R}^{\eta \mu} \subset {Q_{\sigma R}^{\mu }}$ for any $R>0$.
\end{claim}
\begin{proof}
    In order for this to hold, we must have the following:
    \begin{equation*}%
            \left\{\begin{array}{rcl} (\eta \mu )^{-1} c_0 R & \leq & \mu^{-1}\sigma  R\\ (\eta \mu )^{-p} (c_0R)^2 & \leq & \mu^{-p} (\sigma R)^2 \end{array}\right.
             \Longleftrightarrow \left\{\begin{array}{rcl}  c_0 & \leq &  \eta \sigma \\  c_0^2 & \leq & \eta^{p}\sigma ^2 \end{array}\right.
    \end{equation*}%
From the restriction $\sigma< \eta <1$ and the choice $c_0=\frac12\sigma \min\{{\eta}, {\eta}^{p/2}\}$, we see that the above two restrictions hold.
\end{proof}

\begin{claim}\label{claim2.6}
    We have $R_n = c_0^n S$ and $\eta^n = \lbr \frac{R_n}{S} \rbr^{\al_1}$ where $\al_1^{-1} = - \log_{\frac{1}{\eta}} c_0$.
\end{claim}
\begin{proof}
    From direct computation, we see that $\frac{1}{\al_1} = -\frac{\log c_0}{\log \frac{1}{\eta}} = \frac{-\log c_0}{-\log \eta} = \frac{\log c_0}{\log \eta}.$ Thus we get
    \begin{equation*}
         \eta^n = \lbr \frac{R_n}{S} \rbr^{\frac{\log \eta}{\log c_0}}
         \Longleftrightarrow  n \log \eta = \frac{\log \eta}{\log c_0} \log \lbr \frac{R_n}{S}\rbr 
         \Longleftrightarrow   \log c_0^n =  \log \lbr \frac{R_n}{S}\rbr 
         \Longleftrightarrow   S c_0^n =  R_n.
    \end{equation*}
This proves the claim.
\end{proof}

\subsubsection{Switching radius}
Since we have to study the interplay between \cref{alt1} and \cref{alt2}, we need to define what is known as the switching radius:
\begin{definition}[Switching Radius]\label{switch_rad}
    With the notation from \cref{iter_const_def}, suppose the hypothesis of \cref{alt2} holds at level $(R_1,\mu_1),(R_2,\mu_2) \ldots (R_i,\mu_i)$, i.e., 
\begin{equation}\label{8.6}
    |\{ Q_{R_i}^{\mu_i} : |\nabla u| < \mu_i/2\}| \geq \nu |Q_{R_i}^{\mu_i}| \txt{and} s \leq \mu_i,
\end{equation}
holds, then applying \cref{alt2}, we conclude
\begin{equation*}%
    |\nabla u| \leq \eta \mu_{i} = \mu_{i+1} \txt{on} Q_{\sigma R_{i}}^{\mu_{i}} \overset{\text{\cref{cyl_incl}}}{\supseteq} Q_{R_{i+1}}^{\mu_{i+1}}.
\end{equation*}%
Continuing this way, we denote  $n_0$ (called the switching number) and radius $R_{n_0}$ (called the switching radius) to be the first instance where one of the conditions from  \cref{8.6} fails, i.e., 
\begin{equation}\label{8.8eqn}
    |\{ Q_{R_{n_0}}^{\mu_{n_0}} : |\nabla u| < \mu_{n_0}/2\}| < \nu |Q_{R_{n_0}}^{\mu_{n_0}}| \txt{or} \mu_{n_0} < s.
\end{equation}
 or both holds for the first time. 
\end{definition}

From \cref{switch_rad}, we see that the following estimates hold:
\begin{itemize}
\item For $n= 1,\ldots,n_0$, due to \cref{alt2} and Claim \ref{cyl_incl} there holds\begin{equation}\label{2.12}
    \sup_{Q_{R_n}^{\mu_n}} |\nabla u| \leq \mu_n = \eta^n \mu_0.
\end{equation}
\item We also have for any $n = 0,1,\ldots, n_0$, 
\begin{equation} \label{claim2.7}
  \sup_{Q_{R_{n}}^{\mu_n}}|\nabla u| \leq \mu_0 \lbr \frac{R_n}{S}\rbr^{\al_1}.\end{equation}

    This follows from \cref{2.12} and \cref{claim2.6}.
    \item If at the switching number $n_0$ it happens that $s> \mu_{n_0}$, then we have following bounds:
\[
    \sup_{Q_{R_{n_0}}^{\mu_{n_0}}} |\nabla u| \leq \mu_{n_0} \txt{and} \mu_{n_0} \leq s \leq s+ \sup_{Q_{R_{n_0}}^{\mu_{i_0}}} |\nabla u| \leq \mu_{n_0-1} +\mu_{n_0} \leq  \left( \frac{1+\eta}{\eta} \right) \mu_{n_0}
\]
which in particular implies that the hypothesis required to apply \cref{lemma3.2} is satisfied after rescaling as in \eqref{w_rescale}. Thus even in this case, we can directly apply \cref{lemma3.2}  to get the conclusion given in \cref{8.10} below, see more details in \cref{remark8.3}.    

\item If  however $s \leq \mu_{n_0}$ and instead the measure condition  from  \cref{8.8eqn} holds which is    the hypothesis required in  \cref{alt1},  then we have,
\begin{equation}\label{8.10}
    \fiint_{Q_{\de^iR_{n_0}}^{\mu_{n_0}}} |\nabla u - (\nabla u)_i|^2 \ dz\leq \ka^i \fiint_{Q_{R_{n_0}}^{\mu_{n_0}}} |\nabla u - (\nabla u)_0|^2 \ dz \overset{\redlabel{8.10a}{a}}{\leq} \ka^i   \mu_{n_0}^2 \txt{for} i=1,2,\ldots,
\end{equation}
where to obtain \redref{8.10a}{a}, we have used  the notation $(\nabla u)_i := \fiint_{Q_{\de^iR_{n_0}}^{\mu_{n_0}}} \nabla u \ dz$ and the bound   $ \sup_{Q_{R_{n_0}}^{\mu_{n_0}}} |\nabla u| \leq \mu_{n_0}$, which  holds due to \cref{2.12}. Moreover, we made use of  $\fiint |f - (f)|^2 \ dz = \inf_{a \in \RR} \fiint |f-a|^2\ dz \leq \fiint |f|^2$.

\end{itemize}

\begin{remark}
    From triangle inequality, for $i = 1,2,\ldots$, we also have 
    \begin{equation}\label{8.10_trian}
    \begin{array}{rcl}
    \fiint_{Q_{\de^{i-1}R_{n_0}}^{\mu_{n_0}}} |\nabla u - (\nabla u)_i|^2 \ dz & \leq & \fiint_{Q_{\de^{i-1}R_{n_0}}^{\mu_{n_0}}} |\nabla u - (\nabla u)_{i-1}|^2 \ dz + \frac{|{Q_{\de^{i-1}R_{n_0}}^{\mu_{n_0}}}|}{|{Q_{\de^{i}R_{n_0}}^{\mu_{n_0}}}|}\fiint_{Q_{\de^{i-1}R_{n_0}}^{\mu_{n_0}}} |\nabla u - (\nabla u)_{i-1}|^2 \ dz\\
    & \overset{\cref{8.10}}{\leq} &  \ka^{i-1} \lbr 1 + \frac{1}{\de^{N+2}} \rbr \mu_{n_0}^2.
    \end{array}
\end{equation}

\end{remark}

\begin{lemma}\label{lemma8.8}
Let $\kappa, \de, \eta$ be as given in \cref{alt1} and \cref{alt2}, then we have the following important consequences from \cref{2.12,8.10}.
\begin{description}
    \descitem{C1}{conc1}  The sequence $\{(\nabla u)_i\}_{i=1}^{\infty}$ from \cref{8.10} is a Cauchy sequence and converges to $\nabla u(z_0)$ where  $z_0=(x_0,t_0)$ is the center of the parabolic cylinders considered in \cref{2.12} and \cref{8.10}.
    \descitem{C2}{conc2} The following decay estimate holds:
    \begin{equation*}%
     |\nabla u(x_0,t_0) - (\nabla u)_i| \leq C \ka^i \mu_{n_0}^2 \txt{for all} i = 1,2\ldots.
 \end{equation*}%
    \descitem{C3}{conc3}  For any $0<\rho<R_{n_0}$ with $(\nabla u)_{\rho} := \fiint_{Q_{\rho}^{\mu_{n_0}}(x_0,t_0)} \nabla u \ dz$, the following decay estimate holds:
    \begin{equation*}%
     |\nabla u(x_0,t_0) - (\nabla u)_{\rho}| \leq C(\de) \ka^{i} \mu_{n_0}^2,
 \end{equation*}%
 where $i \in \ZZ$ is such that $\de^i R_{n_0} \leq \rho \leq \de^{i-1} R_{n_0}$.
    \descitem{C4}{conc4} Let us define $\al_3 := \min \{ \al_1, \al_2/2\}$ where $\al_2 := -\log_{\frac{1}{\eta}} \de$ and $\al_1$ is from \cref{claim2.6},  then  for any $0<\rho<S$,  with $(\nabla u)_{\rho} := \fiint_{Q_{\rho}^{\mu_{n_0}}(x_0,t_0)} \nabla u \ dz$, there holds
 \begin{equation*}%
     |\nabla u(x_0,t_0) - (\nabla u)_{\rho}| \leq C \mu_0 \lbr \frac{\rho}{S}\rbr^{\al_3}.
 \end{equation*}%
    \descitem{C5}{conc5} With $\al_3$ as defined in \descref{conc4}{C4}, for any $0<\rho\leq S$ with $(\nabla u)_{\rho} := \fiint_{Q_{\rho}^{\mu_{n_0}}(x_0,t_0)} \nabla u \ dz$, we also have 
     \begin{equation*}%
         \fiint_{Q_{\rho}^{\mu_{n_0}}} |\nabla u - (\nabla u)_{\rho}|^2 \ dz \leq C \mu_0^2 \lbr \frac{\rho}{S}\rbr^{2\al_3}.
     \end{equation*}%
\end{description}
\end{lemma}
\begin{proof} From simple triangle inequality, we have 
\begin{equation}\label{triang_inq}
    |(\nabla u)_{i+1} - (\nabla u)_i|^2 \leq 2 |\nabla u - (\nabla u)_{i+1}|^2 + 2 |\nabla u - (\nabla u)_i|^2.
\end{equation}
    \begin{description}[leftmargin=*]
        \item[Proof of \descref{conc1}{C1}:] From Lebesgue differentiation theorem, we see that $(\nabla u)_i \overset{i \nearrow \infty}{\longrightarrow} \nabla u(z_0)$ provided $z_0$ is a Lebesgue point. So all that remains to show that is that the sequence $(\nabla u)_i$ is Cauchy which follows from the following sequence of estimates noting that $\ka \in (0,1)$: 
        \begin{equation}\label{8.16}
    \begin{array}{rcl}
    |(\nabla u)_{i+1} - (\nabla u)_i|^2 & \overset{\cref{triang_inq}}{\leq} & 2 \fiint_{Q_{\de^{i+1}R_{n_0}}^{\mu_{n_0}}}  |\nabla u - (\nabla u)_{i+1}|^2 \ dz + 2 \frac{|Q_{\de^iR_{n_0}}^{\mu_{n_0}}|}{|Q_{\de^{i+1}R_{n_0}}^{\mu_{n_0}}|} \fiint_{Q_{\de^iR_{n_0}}^{\mu_{n_0}}} |\nabla u - (\nabla u)_i|^2 \ dz \\
    & \leq &  2 \ka^{i+1} \mu_{n_0}^2 + \frac{2}{\de^{N+2}} \ka^i \mu_{n_0}^2  = 2 \ka^i \mu_{n_0}^2 \lbr \ka + \frac{1}{\de^{N+2}} \rbr.
    \end{array}
\end{equation}

        \item[Proof of \descref{conc2}{C2}:] By adding and subtracting, for any $j \geq 1$, we have
\begin{equation*}
    \begin{array}{rcl}
        |(\nabla u)_{i+j} - (\nabla u)_i|^2 & \overset{\cref{8.16}}{\leq} & 2 \ka^i \mu_{n_0}^2 \lbr \ka + \frac{1}{\de^{N+2}} \rbr ( 1 + \ka + \ka^2 + \ldots + \ka^j) \\
        & \leq & 2 \ka^i \mu_{n_0}^2 \lbr \ka + \frac{1}{\de^{N+2}} \rbr \frac{1}{1-\ka} 
         =  C(\ka,\de) \ka^i \mu_{n_0}^2.
    \end{array}
\end{equation*}
 In particular, letting $j \rightarrow \infty$, the following holds:
 \begin{equation*}
     |\nabla u(x_0,t_0) - (\nabla u)_i|^2 \leq C \ka^i \mu_{n_0}^2 \txt{for all} i = 1,2\ldots.
 \end{equation*}%
        \item[Proof of \descref{conc3}{C3}:]  Let $\rho < R_{n_0}$ be given, then there exists $i \in \ZZ$ such that 
    \begin{equation}\label{8.18}\de^i R_{n_0} \leq \rho \leq \de^{i-1} R_{n_0},\end{equation} holds, where $\de$ is the constant from \cref{alt1}. Thus, we have the following sequence of estimates:
    \begin{equation}\label{8.19}
     \begin{array}{rcl}
         |(\nabla u)_{\rho} - (\nabla u)_i|^2 & \leq & \fiint_{Q_{\rho}^{\mu_{n_0}}} |\nabla u - (\nabla u)_i|^2 \ dz 
           \leq  \frac{|Q_{\de^{i-1}R_{n_0}}^{\mu_{n_0}}|}{|Q_{\rho}^{\mu_{n_0}}|} \fiint_{Q_{\de^{i-1}R_{n_0}}^{\mu_{n_0}}} |\nabla u - (\nabla u)_i|^2 \ dz \\
          & \overset{\redlabel{8.10a}{a}}{\leq} & \frac{[\de^{i-1} R_{n_0}]^{N+2}}{\rho^{N+2}} C \ka^{i-1} \mu_{n_0}^2 
          \overset{\cref{8.18}}{\leq}  \frac{C}{\de^{N+2}} \ka^{i-1} \mu_{n_0}^2,
     \end{array}
 \end{equation}
 where to obtain \redref{8.10a}{a}, we made use of \cref{8.10} and \cref{8.10_trian}.
Thus from triangle inequality, we get the desired conclusion:
 \begin{equation}\label{2.10}
         |\nabla u(x_0,t_0) - (\nabla u)_{\rho}|^2 \leq  2 |\nabla u(x_0,t_0) - (\nabla u)_{i}|^2 + 2|(\nabla u)_i - (\nabla u)_{\rho}|^2 
          \leq  C_{(\de,\ka)} \ka^{i} \mu_{n_0}^2.
 \end{equation}

        \item[Proof of \descref{conc4}{C4}:]  We split the proof into two cases, either $\rho \leq R_{n_0}$ or $R_{n_0} \leq \rho \leq S$. 
        \begin{description}
            \item[Case $\rho \leq R_{n_0}$:] In this case, there exists $i \in \ZZ$ such that \cref{8.18} holds. Following the calculation from \cref{claim2.6}, we see that $\ka^i \leq \lbr \frac{\rho}{R_{n_0}} \rbr^{\al_2}$. Using this, we get the following sequence of estimates:
            \begin{equation}\label{8.21}
             \begin{array}{rcl}
                 |\nabla u(x_0,t_0) - (\nabla u)_{\rho}| & \overset{\cref{2.10}}{\leq} & C \ka^{\frac{i}{2}} \mu_{n_0} \leq C \lbr \frac{\rho}{R_{n_0}} \rbr^{\frac{\al_2}{2}} \mu_{n_0} 
                 = C \lbr \frac{\rho}{R_{n_0}} \rbr^{\frac{\al_2}{2}} \eta^{n_0}\mu_0 \\
                 & \overset{\text{\cref{claim2.6}}}{=} & C \lbr \frac{\rho}{R_{n_0}} \rbr^{\frac{\al_2}{2}} \lbr \frac{R_{n_0}}{S} \rbr^{\al_1} \mu_0 \\
                 & \leq & \lbr \frac{\rho}{R_{n_0}} \rbr^{\al_3} \lbr \frac{R_{n_0}}{S} \rbr^{\al_3} \mu_0 \\
                 & = & C \lbr \frac{\rho}{S} \rbr^{\al_3} \mu_0 .
             \end{array}
         \end{equation}
         In the above estimate, we noted that $\rho \leq R_{n_0} \leq R$ and $\al_3 = \min\{ \al_1,\al_2/2\}$.
            \item[Case $R_{n_0} \leq \rho \leq S$:] From the definition of $R_{n_0}$ (see \cref{switch_rad}), we see that there exists some $i \in \{1,2,\ldots, n_0\}$ such that $c_0^i S \leq \rho \leq c_0^{i-1} S$ (recall $c_0 \in (0,1)$ from \cref{cyl_incl}). Thus we get the following sequence of estimates:
            \begin{equation}\label{8.22}
            \begin{array}{rcl}
                |\nabla u(x_0,t_0) - (\nabla u)_{\rho}| & \overset{\text{\cref{alt2}}}{\leq} & 2 \sup_{Q_{R_{i-1}}^{\mu_{i-1}}} |\nabla u| 
                  \overset{\cref{2.12}}{\leq}  2 \eta^{i-1} \mu_0 \\
                 & \overset{\text{\cref{claim2.6}}}{=} & 2 \lbr \frac{R_{i-1}}{S} \rbr^{{\al_1}} \mu_0 \overset{\text{\cref{iter_const_def}}}{=} 2 \lbr \frac{c_0^{i-1} S}{S} \rbr^{{\al_1}} \mu_0\\
                 & \leq  & \frac{2}{c_0^{\al_1}} \lbr \frac{\rho}{S} \rbr^{{\al_1}}  \mu_0 \\
                 & \leq & C \lbr \frac{\rho}{S} \rbr^{\al_3} \mu_0 .
            \end{array}
        \end{equation}
        \end{description}
        \item[Proof of \descref{conc5}{C5}:] If $\rho \geq R_{n_0}$, then the conclusion follows directly from \cref{8.22}. In the case $\rho \leq R_{n_0}$, there exists $i \in \ZZ$ such that \cref{8.18} holds. Using this, we get
        \begin{equation*}%
        \begin{array}{rcl}
            \fiint_{Q_{\rho}^{\mu_{n_0}}} |\nabla u - (\nabla u)_{\rho}|^2 \ dz & \leq & 2 \fiint_{Q_{\rho}^{\mu_{n_0}}} |\nabla u - (\nabla u)_{i}|^2 \ dz + 2  |(\nabla u)_i - (\nabla u)_{\rho}|^2 \\
            & \apprle & \frac{| Q_{\de^{i-1}R_{n_0}}^{\mu_{n_0}}|}{| Q_{\rho}^{\mu_{n_0}}|}\fiint_{ Q_{\de^{i-1}R_{n_0}}^{\mu_{n_0}}} |\nabla u - (\nabla u)_{i-1}|^2 \ dz +  |(\nabla u)_{i-1} - (\nabla u)_{\rho}|^2   + 2  |(\nabla u)_i - (\nabla u)_{\rho}|^2 \\ 
            & \overset{\redlabel{8.23a}{a}}{\leq} & \frac{C}{\de^{N+2}} \ka^{i-1} \mu_{n_0}^2 \\
            & \overset{\cref{8.21}}{\leq} & C  \lbr \frac{\rho}{S} \rbr^{2\al_3} \mu_0 ^2,
        \end{array}
    \end{equation*}%
    where to obtain \redref{8.23a}{a}, we made use of \cref{8.10} along with \cref{8.19}.
    \end{description}
    This completes the proof of the lemma. 
    \end{proof}
\subsubsection{Proof of gradient H\"older continuity in time in the cylinder  \texorpdfstring{$Q_{R}^{\mu}$}.}
Let us fix any  point $z_0= (x_0,t_0) \in Q_0$ and let  $\tz_1=(x_0,t_1) \in Q_0$ be given. In this case, we assume that the two points $z_0, \tz_1  \in Q_S^{\mu_0}$ belong to the same cylinder for some cylinder $Q_S^{\mu_0}$. We will prove that the gradient is H\"older continuous at $z_0$ and since the point $z_0 \in Q_0$ is arbitrary, this proves H\'older regularity in cylinders of the form $Q_S^{\mu_0}$.

\begin{remark}
From \cref{switch_rad}, let us denote $\mu_{n_0}$ (respectively $\mu_{n_1}$) to be the switching number and $R_{n_0}$ (respectively $R_{n_1}$) to be the switching radius corresponding to the point $z_0$ (respectively $\tz_1$).  Note that even though these two switching numbers depend on the point, all the estimates and constants in \cref{lemma8.8} are independent of the point. 
\end{remark}

Let $\rho:=d(z_0,\tz_1)$, then from triangle inequality, we have 
      \begin{equation}\label{8.26}
        |\nabla u(x_0,t_0) - \nabla u(x_0,t_1)| \leq |\nabla u(x_0,t_0) - (\nabla u)_{\C_0}|+|(\nabla u)_{\C_1} - \nabla u(x_0,t_1)|+|(\nabla u)_{\C_0} - (\nabla u)_{\C_1}|,
\end{equation}
where we have used the notation $(\nabla u)_{\C_i} := \fiint_{\C_i} \nabla u \ dz$ and $\C_0,\C_1$ are cylinders that will eventually be chosen. In what follows, we shall use $i\in \{0,1\}$ to denote quantities which are related to $z_0$ and $z_1$ respectively. 

 Without loss of generality, let us assume
    \begin{equation}\label{ass_1}
        \mu_{n_0} \leq \mu_{n_1}.
    \end{equation}
    
\begin{description}[leftmargin=*]
    \item[\underline{Case $\min\{\mu_{n_0}^{-p/2}R_{n_0},\mu_{n_1}^{-p/2}R_{n_1}\} \geq 2\rho$:}]  Let $\rho = \sqrt{|t_0-t_1|}$ and  set $S_i:= \mu_{n_i}^{p/2} \sqrt{|t_0-t_1|} = \mu_{n_i}^{p/2} \rho$.  Now let us consider the cylinder $\C_0=Q_{S_0}^{\mu_{n_0}}(z_0)$ and $\C_1 = Q_{S_1}^{\mu_{n_1}}(\tz_1)$, then we have the following observations:
        
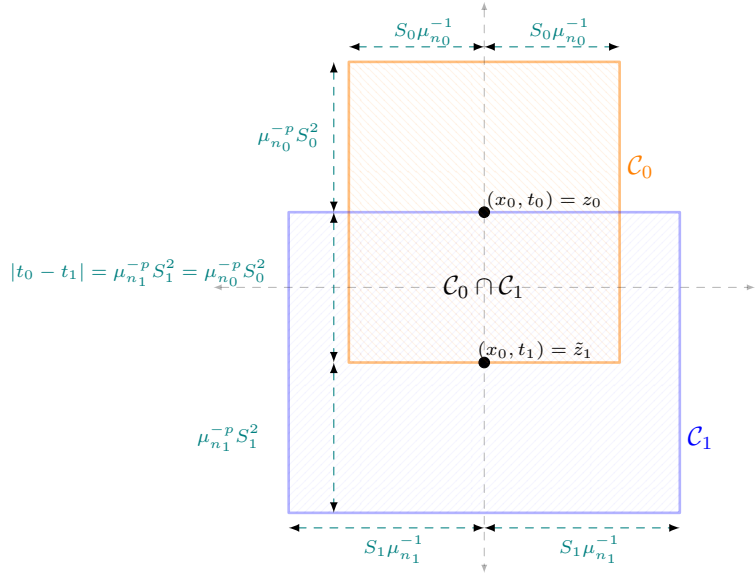
\begin{center}
\begin{tikzpicture}[line cap=round,line join=round,>=latex,x=0.4cm,y=0.4cm]
\draw [line width=1pt,color=orange,opacity=0.5] (-5.5,6) rectangle (3.5,-4);
\draw [line width=1pt,color=blue,opacity=0.3] (-7.5,1) rectangle (5.5,-9);
\draw [line width=0.1pt,opacity=0.3,dashed,<->] (-1,8) -- (-1,-11);
\draw [line width=0.1pt,opacity=0.3,dashed,<->] (-10,-1.5) -- (8,-1.5);
\draw[color=orange] (4.2,2.5) node {$\C_0$};
\draw[color=blue] (6.2,-6.5) node {$\C_1$};
\draw[draw=teal, dashed, <->] (-6,6) -- (-6,1);
\draw[draw=teal, dashed, <->] (-6,1) -- (-6,-4);
\draw[draw=teal, dashed, <->] (-6,-4) -- (-6,-9);
\draw[draw=teal, dashed, <->] (-7.5,-9.5) -- (-1,-9.5);
\draw[draw=teal, dashed, <->] (-1,-9.5) -- (5.5,-9.5);
\draw[draw=teal, dashed, <->] (-5.5,6.5) -- (-1,6.5);
\draw[draw=teal, dashed, <->] (-1,6.5) -- (3.5,6.5);
\draw[color=black] (-1,-1.5) node {$\C_0\cap \C_1$};
\begin{scriptsize}
\draw [fill=black] (-1,1) circle (2pt);
\draw[color=black] (1,1.4) node {$(x_0,t_0)=z_0$};
\draw [fill=black] (-1,-4) circle (2pt);
\draw[color=black] (0.7,-3.6) node {$(x_0,t_1)=\tz_1$};
\draw[color=teal] (-3,7) node {$S_0 \mu_{n_0}^{-1} $};
\draw[color=teal] (1.5,7) node {$S_0 \mu_{n_0}^{-1} $};
\draw[color=teal] (-4,-10.2) node {$S_1 \mu_{n_1}^{-1} $};
\draw[color=teal] (2.5,-10.2) node {$S_1 \mu_{n_1}^{-1} $};
\draw[color=teal] (-7.5,3.5) node {$\mu_{n_0}^{-p} S_0^2$};
\draw[color=teal] (-12.5,-1) node {$|t_0-t_1|=\mu_{n_1}^{-p} S_1^2 = \mu_{n_0}^{-p} S_0^2$};
\draw[color=teal] (-9.5,-6.5) node {$\mu_{n_1}^{-p} S_1^2$};
\end{scriptsize}
\draw [pattern=north west lines, pattern color=orange, opacity=0.3,draw=none] (-5.5,6) rectangle (3.5,-4);
\draw [pattern=north east lines, pattern color=blue, opacity=0.2,draw=none] (-7.5,1) rectangle (5.5,-9);
\end{tikzpicture}
\end{center}
        
        \begin{itemize}
            \item $\C_0,\C_1 \in Q_S^{\mu_0}$. In order to see this, we need to show $S_i\mu_{n_i}^{-1} \leq \mu_0^{-1} S$ and $\mu_{n_i}^{-p} S_i^2 \leq \mu^{-p} S^2$. The space inclusion holds due to the following calculations:
            \[
                S_i\mu_{n_i}^{-1} \leq \mu_0^{-1} S \Longleftrightarrow \rho \leq \lbr \frac{\eta}{c_0}\rbr^{n_i} \mu_{n_i}^{-\frac{p}{2}}R_{n_i} \Longleftarrow \rho \leq \lbr \frac{2}{\sigma}\rbr^{n_i} \mu_{n_i}^{-\frac{p}{2}}R_{n_i},
            \]
and the last inequality holds true since we are in the case $2\rho \leq \mu_{n_i}^{-\frac{p}{2}}R_{n_i}$ where  $\sigma \in (0,1)$. The time inclusion is trivial since $\mu_{n_i}^{-p}S_i^2 = |t_0-t_1| = \rho^2 \leq \mu_0^{-p} S^2$ which holds true since $z_0,z_1 \in Q_S^{\mu_0}$.

            \item Clearly, we see that $\C_0 \cap \C_1 \neq \emptyset$ and we additionally have the following:
            \begin{equation*}
            |\C_0 \cap \C_1| \approx \min\{\mu_{n_0}^{\frac{p}{2}-1},\mu_{n_1}^{\frac{p}{2}-1}\}^N |t_0-t_1|^{\frac{N+2}{2}}, \quad  |\C_0| \approx  \mu_{n_0}^{N\lbr\frac{p}{2}-1\rbr} |t_0-t_1|^{\frac{N+2}{2}}, \quad |\C_1| \approx  \mu_{n_1}^{N\lbr\frac{p}{2}-1\rbr} |t_0-t_1|^{\frac{N+2}{2}}.
            \end{equation*}
        \end{itemize}
             Thus, making use of \cref{ass_1}, we additionally observe the following:
            \begin{equation}\label{area_estimate_time}
            \begin{array}{rcl}
                    \frac{|\C_0|}{|\C_0 \cap \C_1|} & = &   \frac{\mu_{n_0}^{\frac{N(p-1)}{2}} \mu_{n_0}^{-\frac{N}{2}}}{\min \left\{\mu_{n_0}^{\frac{N(p-1)}{2}} \mu_{n_0}^{-\frac{N}{2}},\mu_{n_1}^{\frac{N(p-1)}{2}} \mu_{n_1}^{-\frac{N}{2}} \right\}} \overset{\cref{ass_1}}{\leq}  = \lbr \frac{\mu_{n_1}}{\mu_{n_0}}\rbr^{\frac{N}{2}}, \\
                    \frac{|\C_1|}{|\C_0 \cap \C_1|}&  = &   \frac{\mu_{n_1}^{\frac{N(p-1)}{2}} \mu_{n_1}^{-\frac{N}{2}}}{\min \left\{\mu_{n_0}^{\frac{N(p-1)}{2}} \mu_{n_0}^{-\frac{N}{2}},\mu_{n_1}^{\frac{N(p-1)}{2}} \mu_{n_1}^{-\frac{N}{2}} \right\}} \overset{\cref{ass_1}}{\leq}   \lbr \frac{\mu_{n_1}}{\mu_{n_0}}\rbr^{\frac{N(p-1)}{2}}.
                    \end{array}
            \end{equation}

        Thus the first two terms on the right hand side of \cref{8.26} is estimated as follows (recall notation from  \cref{mua}):
        \begin{equation}\label{8.27_time}
\begin{array}{rcl}
    |\nabla u(x_0,t_i) - (\nabla u)_{\C_i}| & \overset{\text{\descref{conc4}{C4}}}{\leq} &  C \mu_0 \lbr \frac{S_i}{S}\rbr^{\al_3}  = \  C \mu_0 \lbr \frac{\mu_{n_i}^{p/2} \rho}{S}\rbr^{\al_3} \\
    & \overset{\text{\descref{obs2}{O2}}}{\leq} &  C \mu_0^a \lbr \frac{\sqrt{|t_0-t_1|}}{\min\{\mu_0,\mu_0^{p/2}\} R_0}\rbr^{\al_3}
     \leq  C \mu_0^a \lbr \frac{\sqrt{|t_0-t_1|}}{ R_0}\rbr^{\al_3}.
    \end{array}
\end{equation}
The third term on the right hand side of \cref{8.26} is estimated as follows:
\begin{equation}\label{8.28_time}
    \begin{array}{rcl}
        |(\nabla u)_{\C_0} - (\nabla u)_{\C_1}| & \leq & \fiint_{\C_0\cap \C_1} |\nabla u(z) - (\nabla u)_{\C_0}| \ dz + \fiint_{\C_0\cap \C_1} |\nabla u(z) - (\nabla u)_{\C_1}| \ dz\\
        & \leq & \frac{|\C_0|}{|\C_0 \cap \C_1|}  \fiint_{\C_0} |\nabla u(z) - (\nabla u)_{\C_0}| \ dz + \frac{|\C_1|}{|\C_0 \cap \C_1|}\fiint_{\C_1} |\nabla u(z) - (\nabla u)_{\C_1}| \ dz\\
        & \overset{\text{\descref{conc5}{C5}},\cref{area_estimate_time}}{\leq}& C \mu_0 \lbr \frac{\mu_{n_1}}{\mu_{n_0}}\rbr^{\frac{N}{2}}\lbr \frac{S_0}{S} \rbr^{\al_3} +C \mu \lbr \frac{\mu_{n_1}}{\mu_{n_0}}\rbr^{\frac{N(p-1)}{2}} \lbr \frac{S_1}{S} \rbr^{\al_3} \\
        & \overset{\cref{8.27_time}}{\leq} & C \mu_0^a \lbr \lbr \frac{\mu_{n_1}}{\mu_{n_0}}\rbr^{\frac{N(p-1)}{2}} + \lbr \frac{\mu_{n_1}}{\mu_{n_0}}\rbr^{\frac{N}{2}}\rbr \lbr \frac{\sqrt{|t_0-t_1|}}{ R_0}\rbr^{\al_3} \\
    \end{array}
\end{equation}

\item[Removing dependence on $\{\mu_{n_i}\}$:] The estimates in  \cref{8.28_time} still contain $\mu_{n_0}$ and $\mu_{n_1}$ which we control as follows. Let  $\ga \in (0,1)$  be a constant satisfying
\begin{equation*}
\ga \leq \frac{\al_3}{2N} \frac{1}{\max\{p-1,1\}}.
\end{equation*}%
We then  consider the following two cases, either $\mu_{n_0} \geq \mu_{n_1} \lbr \frac{\rho}{R_0}\rbr^{\ga}$ or $\mu_{n_0} \leq \mu_{n_1} \lbr \frac{\rho}{R_0}\rbr^{\ga}$. \textit{Let us first obtain the H\"older continuity of the gradient in the case }
\begin{equation}\label{case1}\mu_{n_0} \geq \mu_{n_1} \lbr \frac{\rho}{R_0}\rbr^{\ga}.\end{equation}
Thus making use of \cref{case1} along with \cref{8.27_time} and \eqref{8.28_time}, we get the following estimates:
\begin{equation*}\begin{array}{rcl}
    |\nabla u(x_0,t_0) - \nabla u(x_0,t_1)| &\leq& C \mu_0^a \lbr \frac{\sqrt{|t_0-t_1|}}{ R_0}\rbr^{\al_3 -\frac{\ga N}{2} \max\{p-1,1\}}, \\
\end{array}\end{equation*}

  \textit{Now we consider the case \cref{case1} fails, then we see that $z_0,\tz_1$ both belong to $\C_0$, since we are in the case $2\rho \leq \min\{\mu_{n_0}^{-p/2}R_{n_0},\mu_{n_1}^{-p/2}R_{n_1}\}$. Thus, we have}
  \[
    |\nabla u(x_0,t_0) - \nabla u(x_0,t_1)| \leq 2\mu_{n_0} \overset{\text{\cref{case1} fails}}{\leq} 2 \mu_{n_1} \lbr \frac{\rho}{R_0}\rbr^{\gamma} \leq 2 \mu_0 \lbr \frac{\sqrt{|t_0-t_1|}}{R_0}\rbr^{\gamma}
  \]
    \item[\underline{Case $\max\{\mu_{n_0}^{-p/2}R_{n_0},\mu_{n_1}^{-p/2}R_{n_1}\} \leq 2\rho$:}] In this case, we directly proceed as follows  (recall notation from  \cref{mua}):
    \begin{equation*}%
        \begin{array}{rcl}
            |\nabla u(z_0) - \nabla u(\tz_1)| & \leq &  \sup_{Q_{R_{n_0}}^{\mu_{n_0}}(z_0)} |\nabla u| + \sup_{Q_{R_{n_1}}^{\mu_{n_1}}(\tz_1)} |\nabla u|         \overset{\cref{2.12}}{\leq}   \eta^{n_0} \mu_0 + \eta^{n_1} \mu_0\\
            & \overset{\text{\cref{claim2.6}}}{=} &  \mu_0  \lbr \frac{R_{n_0}}{S} \rbr^{\al_1} + \mu_0  \lbr \frac{R_{n_1}}{S} \rbr^{\al_1}\leq  C \mu_0^a \lbr \frac{\rho}{S} \rbr^{\al_1} \\
            & \overset{\text{\descref{obs2}{O2}}}{\leq} & C \mu_0^{a}  \lbr \frac{\sqrt{|t_0-t_1|}}{R_0} \rbr^{\al_1}.
        \end{array}
    \end{equation*}%

    \item[\underline{Case $\min\{\mu_{n_0}^{-p/2}R_{n_0},\mu_{n_1}^{-p/2}R_{n_1}\} \leq 2\rho \leq \max\{\mu_{n_0}^{-p/2}R_{n_0},\mu_{n_1}^{-p/2}R_{n_1}\}$:}] Recalling \cref{cyl_incl} and \cref{ass_1}, we see that this case becomes  $\mu_{n_0}^{-p/2}R_{n_0} \leq 2\rho \leq \mu_{n_1}^{-p/2}R_{n_1}$. Let $n_0 \leq n_{\ast} \leq n_1$ be a number such that $\mu_{n_{\ast}+1}^{-p/2}R_{n_{\ast}+1} \leq 2\rho \leq \mu_{n_{\ast}}^{-p/2}R_{n_{\ast}}$ holds. Recalling \cref{cyl_incl}, we then  have 
    \[
        2\rho \leq \min\{\mu_{n_{\ast}}^{-p/2} R_{n_{\ast}}, \mu_{n_1}^{-p/2}R_{n_1}\}.
    \]
    Thus we can replace $\C_0$ in the first case with $\C_{\ast}:= Q_{R_{n_\ast}}^{\mu_{n_\ast}}(z_0)$ and make use of \cref{claim2.7} followed by going through the calculations of the first case to obtain the desired regularity.  More precisely, we argue as follows.    
    
    The first term on the right hand side of \cref{8.26} is estimated as follows:
    \begin{equation}\label{first_term}
        |\nabla u(x_0,t_0) - \avgs{\nabla u}{\C_{\ast}}| \leq 2 \mu_{n_{\ast}} \overset{\text{\cref{claim2.7}}}{\leq} C \mu_0 \lbr \frac{R_{n_{\ast}}}{S}\rbr^{\al_1} \leq C \mu_0^a \lbr \frac{\mu_{n_{\ast}+1}^{p/2}\rho}{S} \rbr^{\al_1} \leq  C \mu_0^a \lbr \frac{\rho}{S}\rbr^{\al_1}.\end{equation}
    The second term on the right hand side of \cref{8.26} is estimated exactly as \cref{8.27_time}. In order to estimate the last term on the right hand side of \cref{8.26}, we make the following observations:
     \begin{equation}\label{all_cases}
                    \frac{|\C_1|}{|\C_{\ast} \cap \C_1|}  \apprle   \frac{\mu_{n_1}^{\frac{N(p-2)}{2}} }{\min \left\{\mu_{n_1}^{\frac{N(p-2)}{2}} ,\mu_{n_{\ast}}^{\frac{N(p-2)}{2}}  \right\}}   \apprle \lbr \frac{\mu_{n_1}}{\mu_{n_{\ast}}}\rbr^{\frac{N(p-1)}{2}},
            \end{equation}
            where to obtain the last estimate, we used the fact that $n_{\ast} \leq n_1$ and proceeded as in \cref{area_estimate_time}.
Now we estimate the last term on the right hand side of \cref{8.26} as follows:
\begin{equation*}
    \begin{array}{rcl}
        |(\nabla u)_{\C_{\ast}} - (\nabla u)_{\C_1}| & \leq & \fiint_{\C_{\ast}\cap \C_1} |\nabla u(z) - (\nabla u)_{\C_\ast}| \ dz + \fiint_{\C_\ast\cap \C_1} |\nabla u(z) - (\nabla u)_{\C_1}| \ dz\\
        & \leq & 2\mu_{n_{\ast}} + \frac{|\C_1|}{|\C_\ast \cap \C_1|}\fiint_{\C_1} |\nabla u(z) - (\nabla u)_{\C_1}| \ dz\\
        & \overset{\redlabel{8.33a}{a}}{\leq}& C \mu_0 \lbr \frac{\rho}{S}\rbr^{\al_1} +C \lbr \frac{\mu_{n_1}}{\mu_{n_\ast}}\rbr^{{\frac{N(p-1)}{2}}} \lbr \frac{S_1}{S} \rbr^{\al_3} \\
        & \overset{\cref{8.27_time}}{\leq} & C \mu_0^a \lbr 1 + \lbr \frac{\mu_{n_1}}{\mu_{n_\ast}}\rbr^{\frac{N(p-1)}{2}}\rbr \lbr \frac{\sqrt{|t_0-t_1|}}{ R_0}\rbr^{\al_3},
    \end{array}
\end{equation*}
where to obtain \redref{8.33a}{a}, we made use of \cref{first_term} along with \cref{all_cases} and \descref{conc5}{C5}. From here onwards, we can proceed as the first case  to remove the dependence of the estimate on $\lbr \frac{\mu_{n_1}}{\mu_{n_\ast}}\rbr^{\frac{N(p-1)}{2}}$.
\end{description}
\subsubsection{Proof of gradient H\"older continuity in space in the cylinder  \texorpdfstring{$Q_{R}^{\mu}$}.}
Let us fix any  point $z_0= (x_0,t_0) \in Q_0$ and let  $\tz_1=(x_1,t_0) \in Q_0$ be given. In this case, we assume that the two points $z_0, \tz_1  \in Q_S^{\mu_0}$ belong to the same cylinder for some cylinder $Q_S^{\mu_0}$. We will prove that the gradient is H\"older continuous at $z_0$ and since the point $z_0 \in Q_0$ is arbitrary, this proves H\'older regularity in cylinders of the form $Q_S^{\mu_0}$.

Let $\rho:=d(z_0,\tz_1)$, we have the same estimate as \cref{8.26} and  without loss of generality, let us assume \cref{ass_1} holds. Since the proof is very similar to the time case, we only give a rough sketch. 
    \begin{description}[leftmargin=*]
            \item[\underline{Case $\min\{\mu_{n_0}^{-1}R_{n_0},\mu_{n_1}^{-1}R_{n_1}\} \geq 2\rho$:}] Let us set $\rho = |x_0-x_1|$ and $S_i := \mu_{n_i} |x_0-x_1| = \mu_{n_i} \rho$ and construct the cylinders $\C_0=Q_{S_0}^{\mu_{n_0}}(z_0)$ and $\C_1 = Q_{S_1}^{\mu_{n_1}}(\tz_1)$, then we have the following observations:
            
            
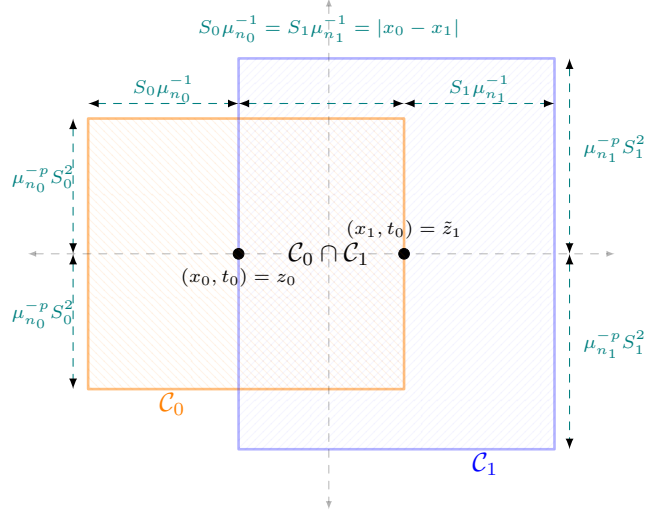
\begin{center}
\begin{tikzpicture}[line cap=round,line join=round,>=latex,x=0.4cm,y=0.4cm]
\draw [line width=1pt,color=orange,opacity=0.5] (-9,3) rectangle (1.5,-6);
\draw [line width=1pt,color=blue,opacity=0.3] (-4,5) rectangle (6.5,-8);
\draw [line width=0.1pt,opacity=0.3,dashed,<->] (-1,7) -- (-1,-10);
\draw [line width=0.1pt,opacity=0.3,dashed,<->] (-11,-1.5) -- (8.5,-1.5);
\draw[color=orange] (-6.2,-6.5) node {$\C_0$};
\draw[color=blue] (4.2,-8.5) node {$\C_1$};
\draw[draw=teal, dashed, <->] (-9,3.5) -- (-4,3.5);
\draw[draw=teal, dashed, <->] (-4,3.5) --(1.5,3.5);
\draw[draw=teal, dashed, <->] (1.5,3.5) -- (6.5,3.5);
\draw[draw=teal, dashed, <->] (-9.5,3) -- (-9.5,-1.5);
\draw[draw=teal, dashed, <->] (-9.5,-1.5) -- (-9.5,-6);
\draw[draw=teal, dashed, <->] (7,5) -- (7,-1.5);
\draw[draw=teal, dashed, <->] (7,-1.5) -- (7,-8);
\draw[color=black] (-1,-1.5) node {$\C_0\cap \C_1$};
\begin{scriptsize}
\draw [fill=black] (-4,-1.5) circle (2pt);
\draw[color=black] (-4,-2.3) node {$(x_0,t_0)=z_0$};
\draw [fill=black] (1.5,-1.5) circle (2pt);
\draw[color=black] (1.5,-0.7) node {$(x_1,t_0)=\tz_1$};
\draw[color=teal] (-6.5,4) node {$S_0 \mu_{n_0}^{-1} $};
\draw[color=teal] (-1,6) node {$S_0 \mu_{n_0}^{-1}= S_1\mu_{n_1}^{-1} = |x_0-x_1| $};
\draw[color=teal] (4,4) node {$S_1 \mu_{n_1}^{-1} $};
\draw[color=teal] (-10.5,1) node {$\mu_{n_0}^{-p} S_0^2$};
\draw[color=teal] (-10.5,-3.5) node {$\mu_{n_0}^{-p} S_0^2$};
\draw[color=teal] (8.5,2) node {$\mu_{n_1}^{-p} S_1^2$};
\draw[color=teal] (8.5,-4.5) node {$\mu_{n_1}^{-p} S_1^2$};
\end{scriptsize}
\draw [pattern=north west lines, pattern color=orange, opacity=0.3,draw=none] (-9,3) rectangle (1.5,-6);
\draw [pattern=north east lines, pattern color=blue, opacity=0.2,draw=none] (-4,5) rectangle (6.5,-8);
\end{tikzpicture}
\end{center}
\begin{itemize}
\item In order for $\C_i \in Q_R^{\mu}$, we need $S_i\mu_{n_i}^{-1} \leq \mu_0^{-1} S$ and $\mu_{n_i}^{-p} S_i^2 \leq \mu_0^{-p} S^2$. The time inclusion holds due to the following calculations:
            \[
                S_i^2\mu_{n_i}^{-p} \leq \mu_0^{-p} S^2 \Longleftrightarrow \rho \leq \lbr \frac{\eta^{p/2}}{c_0}\rbr^{n_i} \mu_{n_i}^{-1}R_{n_i} \Longleftarrow \rho \leq \lbr \frac{2}{\sigma}\rbr^{n_i} \mu_{n_i}^{-1}R_{n_i},
            \]
and the last inequality holds true since we are in the case $2\rho \leq \mu_{n_i}^{-1}R_{n_i}$ where  $\sigma \in (0,1)$. The space inclusion is trivial since $\mu_{n_i}^{-1}S_i = |x_0-x_1| = \rho \leq \mu_0^{-1} S^2$ which holds true since $z_0,\tz_1 \in Q_S^{\mu_0}$.

\item The analogue of \cref{area_estimate_time} becomes 
\begin{equation*}
            \begin{array}{rcl}
                    \frac{|\C_0|}{|\C_0 \cap \C_1|} & = &   \frac{\mu_{n_0}^{1-p} \mu_{n_0}}{\min \left\{\mu_{n_0}^{1-p} \mu_{n_0},\mu_{n_1}^{{1-p}} \mu_{n_1} \right\}} \overset{\cref{ass_1}}{\leq}  = \lbr \frac{\mu_{n_1}}{\mu_{n_0}}\rbr^{p-1}, \\
                    \frac{|\C_1|}{|\C_0 \cap \C_1|}&  = &  \frac{\mu_{n_1}^{1-p} \mu_{n_1}}{\min \left\{\mu_{n_0}^{1-p} \mu_{n_0},\mu_{n_1}^{{1-p}} \mu_{n_1} \right\}} \overset{\cref{ass_1}}{\leq}   \lbr \frac{\mu_{n_1}}{\mu_{n_0}}\rbr.
                    \end{array}
            \end{equation*}
            \item Rest of the calculations goes through verbatim as in the time case.
\end{itemize}

                \item[\underline{Case $\max\{\mu_{n_0}^{-1}R_{n_0},\mu_{n_1}^{-1}R_{n_1}\} \leq 2\rho$:}] This case follows exactly as the time case. 
                   \item[\underline{Case $\min\{\mu_{n_0}^{-1}R_{n_0},\mu_{n_1}^{-1}R_{n_1}\} \leq 2\rho \leq \max\{\mu_{n_0}^{-1}R_{n_0},\mu_{n_1}^{-1}R_{n_1}\}$:}] Recalling \cref{cyl_incl} and \cref{ass_1}, this case becomes  $\mu_{n_0}^{-1}R_{n_0} \leq 2\rho \leq \mu_{n_1}^{-1}R_{n_1}$.  We can now replace $\C_0$ with $\C_{\ast} = Q_{R_{n_{\ast}}}^{\mu_{n_{\ast}}}(z_0)$ where $n_0 \leq n_{\ast} \leq n_1$ is a number satisfying $\mu_{n_{\ast}+1}^{-1}R_{n_{\ast}+1} \leq 2\rho \leq \mu_{n_{\ast}}^{-1}R_{n_{\ast}}$.  Rest of the calculations go through  exactly as in the time case.
    \end{description}

\subsubsection{Proof of gradient H\"older continuity in \texorpdfstring{$Q_0$}.}

In the previous subsection, we proved gradient H\"older continuity at any two points  provided both of them belonged to $z_0,\tz_1 \in Q_S^{\mu_0}$. In this subsection, we consider the case where both of the points does not belong to a single $Q_S^{\mu_0}$, i.e., if we consider the cylinder $Q_S^{\mu_0}(z_0)$, we are in the case $\tz_1 \notin Q_S^{\mu_0}(z_0)$. Thus we observe
\begin{equation}
\label{8.35}
    \max\{ \mu_0, \mu_0^{p/2}\} d(z_0,z_1)  \geq S \overset{\text{\descref{obs2}{O2}}}{\geq}\min\{ \mu_0, \mu_0^{p/2}\} 3R_0,
\end{equation}
holds, using which we get
\begin{equation*}%
        \frac{|\nabla u(z_0) - \nabla u(z_1)|}{d(z_0,z_1)}  \overset{\cref{8.35}}{\leq} \frac{2\mu_0}{R_0} \frac{    \max\{ \mu_0, \mu_0^{p/2}\}}{    \min\{ \mu_0, \mu_0^{p/2}\}}.
\end{equation*}%
This completes the proof of gradient H\"older continuity.

\subsection{Proof of first alternative - \texorpdfstring{\cref{alt1}}.}
Since $u$ is a weak solution of \cref{main_holder} on $Q_r^{\mu}$ for some $r \in (0,R]$, let us perform the following rescaling: Define
\begin{equation}\label{w_rescale}
        w(x,t) = \frac{u(\mu^{-1}rx,\mu^{-p}r^2t)}{\mu^{-1}r},
    \end{equation}
    then $w$ solves
    \begin{equation}\label{def_w}
        \mu^{p-1} w_t - \mu \dv \aa(\nabla w) = 0 \txt{on} Q_1 := B_1 \times I_1.
    \end{equation}
Let us first prove an energy estimate satisfied by \cref{def_w}.
\begin{lemma}
    \label{energy_w} Let $k \in \RR$ and  $\phi \in C^{\infty}(Q_1)$ be any cut-off function with $\phi = 0$ on $\pa_p(Q_1)$, then the following estimate holds:
\begin{equation*}
            \begin{array}{l}
                \sup_{t \in I_1}\mu^{p-1} \int_{B_1} (w_{x_i} - k)_-^2 \phi^2  \ dx+ {C_0}\mu \iint_{Q_1} (s^2+ |\nabla w|^2)^{\frac{p-2}{2}} |\nabla (w_{x_i} - k)_-|^2 \phi^2 \ dz \\
                \hspace*{3cm} \leq     2\mu^{p-1} \iint_{Q_1} (w_{x_i} - k)_-^2 \phi \phi_t \ dz + \frac{4C_1^2}{C_0}\mu  \iint_{Q_1} (s^2+ |\nabla w|^2)^{\frac{p}{2}} |\nabla  \phi|^2 \lsb{\chi}{\{w_{x_i} \leq k\}}  \ dz    \\
                \hspace*{5cm}  + 2C_1\mu \iint_{Q_1} (s^2+ |\nabla w|^2)^{\frac{p-1}{2}} (w_{x_i} -k)_- (|\phi||\nabla^2\phi| + |\nabla \phi|^2)\ dz.
            \end{array}
        \end{equation*}
\end{lemma}
\begin{proof}
    Let us differentiate \cref{def_w} with respect to $x_i$ for some $i \in\{1,2,\ldots,N\}$ and then take $(w_{x_i}-k)_-\phi^2$ with $\phi \in C_c^{\infty}$  as a test function to get
    \begin{equation*}
            \mu^{p-1} \iint_{Q_1} (w_{x_i})_t (w_{x_i}-k)_-\phi^2\ dz + \mu \iint_{Q_1} \iprod{\partial_i \aa(\nabla w)}{\nabla \phi^2} (w_{x_i}-k)_-\ dz + \mu \iint_{Q_1} \iprod{\partial_i \aa(\nabla w)}{\nabla (w_{x_i}-k)_-} \phi^2 \ dz = 0.
    \end{equation*}
    Recalling the notation from \cref{structure_aa_holder}, we see that $\pa_i \aa(\nabla w) = \aa'(\nabla w) \nabla w_{x_i}$.  Let us estimate each of the terms as follows:
\begin{description}
    \item[Estimate for $I$:] This can be estimated as follows:
    \begin{equation}\label{est_I}
        \mu^{p-1} \iint_{Q_1} (w_{x_i})_t (w_{x_i}-k)_-\phi^2 = \frac{\mu^{p-1}}{2} \iint_{Q_1} \lbr(w_{x_i}-k)_-^2\phi^2\rbr_t \ dz - {\mu^{p-1}} \iint_{Q_1}  (w_{x_i}-k)_-^2 \phi \phi_t \ dz.
    \end{equation}
    \item[Estimate for $II$:] Integrating by parts, we get
    \begin{equation}\label{est_II}
        \begin{array}{rcl}
            \mu \iint_{Q_1} \iprod{\partial \aa(\nabla w)}{\nabla \phi} \phi (w_{x_i}-k)_- \ dz & = & - \mu \iint_{Q_1} \iprod{\aa(\nabla w)}{\pa_{x_i} \nabla \phi^2}  \phi (w_{x_i}-k)_- \ dz\\
            && -\mu \iint_{Q_1} \iprod{ \aa(\nabla w)}{\nabla \phi^2} \pa_{x_i} (w_{x_i}-k)_- \ dz \\
            & \overset{\redlabel{9.11a}{a}}{\leq} & C_1\mu  \iint_{Q_1} (s^2 + |\nabla w|^2)^{\frac{p-1}{2}} (|\nabla ^2\phi| + |\nabla \phi|^2) |(w_{x_i}-k)_-|\ dz \\
            && + C_1\mu\ve \iint_{Q_1} (s^2 + |\nabla w|^2)^{\frac{p-2}{2}} |\phi|^2 |\nabla  (w_{x_i}-k)_-|^2 \ dz\\
            &&+ C_1\frac{\mu}{\ve} \iint_{Q_1} (s^2 + |\nabla w|^2)^{\frac{p}{2}}|\nabla \phi|^2 \ dz,
        \end{array}
    \end{equation}
    where to obtain \redref{9.11a}{a}, we made use of \cref{structure_aa_holder}, the trivial bound $|\phi| \leq 1$ and   Young's inequality.
    \item[Estimate for $III$:] We estimate this term as follows:
    \begin{equation}\label{est_III}
            \mu \iint_{Q_1} \iprod{\partial_i \aa(\nabla w)}{\nabla (w_{x_i}-k)_-} \phi^2 \ dz  \overset{\cref{structure_aa_holder}}{\geq} C_0\mu \iint_{Q_1} (s^2+|\nabla w|^2)^{\frac{p-2}{2}} |\nabla (w_{x_i}-k)_-|^2 \phi^2 \ dz.
    \end{equation}
\end{description}
Combining \cref{est_I}, \cref{est_II} and \cref{est_III} gives the desired estimate provided $\epsilon$ is chosen small enough. 

\end{proof}

\subsubsection{DeGiorgi type iteration for \texorpdfstring{$u_{x_i}$}.}
The main proposition we prove is the following:
\begin{proposition}\label{prop3.1}
    Let $r \in (0,R]$ and assume that 
    \begin{equation}\label{3.3}
        s + \sup_{Q_r^{\mu}} \|\nabla u\| \leq A \mu
    \end{equation}
    holds for some $A \geq 1$. For any $i \in \{1,2,\ldots,N\}$,  there exists universal constant $\nu \in (0,1/2)$ such that if
    \begin{equation*}%
        | \{ (x,t) \in Q_r^{\mu} : u_{x_i} < \mu/2\}| \leq \nu |Q_r^{\mu}|,
    \end{equation*}%
    holds, then we have the following conclusion:
    \begin{equation*}%
        u_{x_i} \geq \frac{\mu}{4} \txt{on} Q_{r/2}^{\mu}.
    \end{equation*}%
\end{proposition}
Following the rescaling from \cref{w_rescale}, we can restate the following equivalent version of \cref{prop3.1} for $w$ as follows:
\begin{lemma}\label{lemma3.1}
    Suppose 
    \begin{equation}\label{3.3_w}
        s + \sup_{Q_1} \|\nabla w\| \leq A \mu
    \end{equation}
    holds for some $A \geq 1$. For any $i \in \{1,2,\ldots,N\}$,  there exists universal constant $\nu \in (0,1/2)$ such that if
    \begin{equation}\label{9.48}
        | \{ (x,t) \in Q_1 : w_{x_i} < \mu/2\}| \leq \nu |Q_1|,
    \end{equation}
    holds, then we have the following conclusion:
    \begin{equation*}
        w_{x_i} \geq \frac{\mu}{4} \txt{on} Q_{1/2}.
    \end{equation*}
\end{lemma}
\begin{proof}
Let us define the following constants:
\begin{equation*}
        k_m:= k_0 - \frac{H}{8(1+A)} \lbr 1 - \frac{1}{2^m}\rbr \txt{where} H:= \sup_{Q_1} (w_{x_i}-k_0)_- \txt{and} k_0:= \frac{\mu}{2}.
    \end{equation*}%
    This choice of $k_m$ satisfy the following bounds:
    \begin{equation}\label{3.18}
        k_m - k_{m+1} \geq \frac{\mu}{2^{m+6}(1+A)}, \qquad k_m \geq \frac{\mu}{4} \txt{and} k_m \rightarrow k_{\infty} = k_0 - \frac{H}{8(1+A)} > \frac{\mu}{4}.
    \end{equation}
    Let us now define the following sequence of dyadic parabolic cylinders:
    \begin{equation*}
        Q_m:= Q_{\rho_m} \txt{where} \rho_m := \frac12 + \frac{1}{2^{m+1}}.
    \end{equation*}
    Note that $Q_m \rightarrow Q_{1/2}$ and $Q_0 = Q_1$. 
    Furthermore, let us consider the following cut-off function $\eta_m \in C^{\infty}(Q_m)$ with $\eta_m = 0$ on $\pa_pQ_m$ and $\eta_m \equiv 1$ on $ Q_{m+1}$. The following bounds hold,
    \begin{equation*}
        |\nabla^ 2\eta_m| + |\nabla\eta_m|^2 + |(\eta_m)_t| \leq C(n) 4^m.
    \end{equation*}%
    We will split the proof of the lemma into several steps.
    \begin{description}[leftmargin=*]
        \ditem{\underline{Step $1$:}}{step1} In this step, we show that without loss of generality, we can assume $4H \geq \mu$. Suppose not, then we would have $4H < \mu$ which implies
        \begin{equation*}
        \sup_{Q_1} (w_{x_i}-k_0)_- = \frac{\mu}{2} - \inf_{Q_1} w_{x_i} < \frac{\mu}{4}.
    \end{equation*}
    This says  $w_{x_i} > \frac{\mu}{4}$ and the desired conclusion of \cref{lemma3.1}  follows.
        \ditem{\underline{Step $2$:}}{step2} In this step, we will apply \cref{par_sob_emb} to obtain an estimate for the level sets. In order to do this, let us define 
        \begin{equation*}
            A_m := \{(x,t) \in Q_m: w_{x_i} < k_m\},
    \end{equation*}%
    and consider the following function
    \begin{equation*}
        \tw_m := \left\{ \begin{array}{ll}
                            0 &\ \  \text{if} \ \ w_{x_i} > k_m, \\
                            k_m - w_{x_i} &\ \  \text{if} \ \ k_m \geq w_{x_i} > k_{m+1}, \\
                            k_m - k_{m+1} &\ \  \text{if} \ \  k_{m+1} \geq w_{x_i}. 
                        \end{array}\right.
    \end{equation*}%
    Since $\eta_m = 1$ on $Q_{m+1}$, we obtain the following sequence of estimates: 
    \begin{equation}\label{3.21}
        \begin{array}{rcl}
            \mu^{p-1} (k_m - k_{m+1})^2 |A_{m+1}| & = & \mu^{p-1} \|\tw_m\|_{L^2(A_{m+1})}^2\\
            & \overset{\redlabel{321a}{a}}{\leq} & \mu^{p-1} \|\tw_m\eta_m\|_{L^2(Q_{m})}^2\\
            & \overset{\redlabel{321b}{b}}{\apprle} &  \mu^{p-1} \|\tw_m\eta_m\|_{V^{2,2}(Q_{m})}^2 |A_m|^{\frac{2}{N+2}},
        \end{array}
    \end{equation}
    where to obtain \redref{321a}{a}, we enlarged the domain and made use of the fact that $\spt(\eta_m) \subset Q_m$ and to obtain \redref{321b}{b}, we made use of \cite[Corollary 3.1 - Page 9]{DB93} noting that $\tw_m \eta_m$ is non-negative.  This additionally also implies
    \[
        |\{ Q_m: \tw_m \eta_m >0 \}| \leq | \{Q_m: \tw_m >0\}| = |A_m|.
    \]
    Then, observing that 
    \[
        \tw_m \leq (w_{x_i} - k_m)_- \txt{and} |\nabla \tw_m| \leq |\nabla (w_{x_i} - k_m)_-|  \lsb{\chi}{Q_1 \setminus \{w_{x_i} < k_{m+1}\}},
    \]
    and recalling \cref{func_space},  we get
    \begin{equation}\label{9.55}
        \begin{array}{rcl}
            \mu^{p-1} \|\tw_m\eta_m\|_{V^2(Q_{m})}^2 & \leq & \sup_{-1<t<0} \mu^{p-1} \int_{B_1} (w_{x_i}-k_m)_-^2 \eta_m^2 \ dx \\
            && \qquad + \mu^{p-1} \iint_{Q_1} |\nabla (w_{x_i} - k_m)_-|^2 \lsb{\chi}{Q_1 \setminus \{w_{x_i} < k_{m+1}\}} \eta_m^2 \ dz\\
            &&\qquad  + \mu^{p-1} \iint_{Q_1} (w_{x_i} - k_m)_-^2 |\nabla \eta_m|^2\ dz.
        \end{array}
    \end{equation}
        \ditem{\underline{Step $3$:}}{step3} In this step, we shall estimate \cref{9.55} and obtain a suitable decay of the level set $A_m$ as follows: From \cref{3.18}, we see that 
        \begin{equation}\label{9.56}
            \mu \leq 4 k_{m+1} \leq 4 w_{x_i} \leq 4|\nabla w| \txt{on} Q_1 \setminus \{w_{x_i} < k_{m+1}\}.
        \end{equation}
        Thus making use if \cref{9.56} into \cref{9.55}, we get
    \begin{equation}\label{3.23}
        \begin{array}{rcl}
            \mu^{p-1} \|\tw_m\eta_m\|_{V^2(Q_{m})}^2 & \leq & \sup_{-1<t<0} \mu^{p-1} \int_{B_1} (w_{x_i}-k_m)_-^2 \eta_m^2 \ dx \\
            && \qquad +  \iint_{Q_1} (s^2 + |\nabla w|^2 )^{\frac{p-1}{2}} |\nabla (w_{x_i} - k_m)_-|^2 \lsb{\chi}{Q_1 \setminus \{w_{x_i} < k_{m+1}\}} \eta_m^2 \ dz\\
            &&\qquad  + \mu^{p-1} \iint_{Q_1} (w_{x_i} - k_m)_-^2 |\nabla \eta_m|^2\ dz.
        \end{array}
    \end{equation}
    From \cref{3.3_w}, we see that 
    \begin{equation}\label{3.24}    
            A \mu \iint_{Q_1} (s^2 + |\nabla w|^2)^{\frac{p-2}{2}} |\nabla (w_{x_i}-k)_-|^2 \phi^2 \ dz \geq  \iint_{Q_1} (s^2 + |\nabla w|^2)^{\frac{p-1}{2}} |\nabla (w_{x_i}-k)_-|^2 \phi^2\ dz.
    \end{equation}
    Thus substituting \cref{3.24} into \cref{3.23} and making use of \cref{energy_w} to estimate each of the terms appearing on the right hand side of \cref{3.23} using \cref{3.3_w}, we get
\begin{equation}\label{3.25}
        \begin{array}{rcl}
            \mu^{p-1} \|\tw_m\eta_m\|_{V^2(Q_{m})}^2 & \leq & \mu \iint_{Q_1}(s^2+ |\nabla w|^2)^{\frac{p}{2}} |\nabla \eta_m|^2\lsb{\chi}{\{w_{x_i} \leq k_m\}}\ dz\\
            && + \mu^{p} \iint_{Q_1} (w_{x_i} - k_m)_- ( |\eta_m||\nabla^2\eta_m| + |\nabla \eta_m|^2)\\
            && + \mu^{p-1} \iint_{Q_1} (w_{x_i} - k_m)_-^2  |\nabla  \eta_m|^2 \\
            & \leq & C \mu^{p+1} 4^m  |A_m|.
\end{array}
\end{equation}
Thus combining \cref{3.25} with \cref{3.21} and making use of \cref{3.18}, we get
\begin{equation*}
    \mu^{p-1}\frac{\mu^2}{4^{m+6} (1+A)^2} |A_{m+1}| \leq C 4^m\mu^{p+1} |A_m|^{1+\frac{2}{N+2}} \Longleftrightarrow |A_{m+1}| \leq C16^m  |A_m|^{1+\frac{2}{N+2}}.
\end{equation*}%
We can now apply \cref{iteration} to obtain the existence of a universal constant $\nu \in (0,1)$ such that if
\[
    |A_0|= |\{Q_1: w_{x_i} < \mu/2\}|  < \nu |Q_1|,
\]
then $|A_m| \rightarrow 0$ as $m \rightarrow \infty$. In particular, this says
\[
    w_{x_i} \geq \frac{\mu}{4} \txt{on} Q_{1/2}.
\]
    \end{description}
This completes the proof of the lemma. 
\end{proof}

Following analogous calculations, we can also obtain the dual version of \cref{prop3.1}.
\begin{proposition}\label{prop_upper_bnd}
    Let $r \in (0,R]$ and assume that 
    \begin{equation*}
        s + \sup_{Q_r^{\mu}} \|\nabla u\| \leq A \mu
    \end{equation*}
    holds for some $A \geq 1$. For any $i \in \{1,2,\ldots,N\}$,  there exists universal constant $\nu \in (0,1/2)$ such that if
    \begin{equation}\label{9.48_alt}
        | \{ (x,t) \in Q_r^{\mu} : u_{x_i} > -\mu/2\}| \leq \nu |Q_r^{\mu}|,
    \end{equation}
    holds, then we have the following conclusion:
    \begin{equation*}
        u_{x_i} \leq -\frac{\mu}{4} \txt{on} Q_{r/2}^{\mu}.
    \end{equation*}
\end{proposition}

\subsubsection{Proof of the decay estimate from applying the linear theory}
Let us first recall the weak Harnack inequality and a decay estimate that will be needed to complete the proof of \cref{alt1}, the details can be found in \cite[Lemma 3.1]{KM}.
\begin{lemma}\label{lemma_linear}
    Let $v \in L^2(-1,1;W^{1,2}(B_1))$ be a weak solution to the linear parabolic equation 
    \[
        v_t - \dv (B(x,t) \nabla v) = 0,
    \]
    where the matrix $B(x,t)$ is bounded, measurable and satisfies
    \[
        C_0 |\zeta|^2 \leq \iprod{B(x,t)\zeta}{\zeta} \txt{and} |B(x,t)| \leq C_1,
    \]
    for any $\zeta \in \RR^N$ and $0<C_0 \leq C_1$ are fixed constants. Then there exists a constant $C = C(N,C_0,C_1) \geq 1$ and $\be = \be(N,C_0,C_1) \in (0,1)$ such that the following estimates are satisfied:
    \begin{gather*}
        \sup_{Q_{1/2}}|v| \leq C \lbr \fiint_{Q_1} |v|^q \ dz \rbr^{\frac{1}{q}} \txt{for any $q \in [1,2]$,}\\
        \lbr \fiint_{Q_{\de}} |v - \avgs{v}{Q_{\de}}|^q \ dz \rbr^{\frac{1}{q}} \leq C \de^{\be} \lbr \fiint_{Q_{1}} |v - \avgs{v}{Q_{1}}|^q \ dz \rbr^{\frac{1}{q}} \txt{whenever $q \in [1,2]$ and $\de \in (0,1)$.}
    \end{gather*}
\end{lemma}
 We now have all the estimates needed to prove \cref{alt1} which will be given the following lemma. The proof follows very closely to \cite[Lemma 3.2]{KM} and hence we will only present the rough sketch of the proof.
 \begin{lemma}\label{lemma3.2}
     Assume that in the cylinder $Q_r^{\mu}$, the following is satisfied for a fixed $A \geq 1$:
     \begin{equation*}%
         0 < \frac{\mu}{4} \leq \|\nabla u\|_{L^{\infty}(Q_r^{\mu})} \leq s +\|\nabla u\|_{L^{\infty}(Q_r^{\mu})} \leq A \mu.
     \end{equation*}%
     Then there exists constants $\be= \be(N,p,C_0,C_1,A) \in (0,1)$ and $C = C(N,p,C_0,C_1,A) \geq 1$ such that the following holds for any $\de \in (0,1)$ and $q \geq 1$:
     \begin{equation*}%
         \lbr \fiint_{Q_{\de r}^{\mu}} |\nabla u - \avgs{\nabla u}{Q_{\de r}^{\mu}}|^q \ dz \rbr^{\frac{1}{q}} \leq C \de^{\be} \lbr \fiint_{Q_{r}^{\mu}} |\nabla u - \avgs{\nabla u}{Q_{r}^{\mu}}|^q \ dz \rbr^{\frac{1}{q}}.
     \end{equation*}%
 \end{lemma}
 \begin{proof}
     Without loss of generality, let us assume $\de \in (0,1/2)$ noting that when $\de \in [1/2,1)$, we can enlarge the domain of integration (see \cite[Lemma 3.2 and Proposition 3.3]{KM} for the details) and obtain the desired conclusion, albeit with a larger constant, the details of which is given in \cref{prop3.3}. As in the proof of \cref{lemma3.1}, we rescale according to \cref{w_rescale}, then the hypothesis becomes 
     \begin{equation}\label{first_alt_hyp}
         0 < \frac{\mu}{4} \leq \|\nabla w\|_{L^{\infty}(Q_1)} \leq s +\|\nabla w\|_{L^{\infty}(Q_1)} \leq A \mu.
     \end{equation}
Differentiating \cref{def_w} and dividing the resulting equation by $\mu^{p-1}$, we see that $w_{x_i}$ solves
\begin{equation*}
    (w_{x_i})_t - \dv (B(x,t) \nabla w_{x_i}) = 0 \txt{where} B(x,t) = \mu^{2-p} \pa \aa(\nabla w).
\end{equation*}%
Furthermore, making use of \cref{structure_aa_holder} along with \cref{first_alt_hyp}, we see that the matrix $B(x,t)$ satisfies the following bounds:
\begin{equation}\label{w_x_i_sol}
    C^{-1} |\zeta|^2 \leq \iprod{B(x,t)\zeta}{\zeta} \leq C \lbr \frac{s+\mu}{\mu} \rbr^{p-2} |\zeta|^2 \leq C |\zeta|^2,
\end{equation}
for any $\zeta \in \RR^N$ and $C= C(N,p,C_0,C_1,A)$. Thus we can apply \cref{lemma_linear} to get the desired conclusion whenever $q \in [1,2]$. On the other hand, if $q \geq 2$, then we can follow the calculations from \cite[Equations (3.33) and (3.34)]{KM} to get the desired conclusion. This completes the proof of the lemma. 
 \end{proof}

 \subsubsection{Summary of the proof of \texorpdfstring{\cref{alt1}}.}
 Collecting all the calculations from the previous subsections, we have the following proposition, the proof of which follows verbatim as in \cite[Proposition 3.3]{KM}.
 \begin{proposition}
     \label{prop3.3}
     Assume that \cref{3.3} is in force, then there exists $\nu = \nu(N,p,C_0,C_1,A)\in (0,1/2)$ such that if there exists $i \in \{1,2,\ldots,N\}$ such that either \cref{9.48} or \cref{9.48_alt} holds, then there exists $\be = \be(N,p,C_0,C_1,A) \in (0,1)$ and $C= C(N,p,C_0,C_1,A)$ such that for any $\de \in (0,1)$, the following conclusions hold:
     \begin{gather*}
                  \lbr \fiint_{Q_{\de r}^{\mu}} |\nabla u - \avgs{\nabla u}{Q_{\de r}^{\mu}}|^q \ dz \rbr^{\frac{1}{q}} \leq C \de^{\be} \lbr \fiint_{Q_{r}^{\mu}} |\nabla u - \avgs{\nabla u}{Q_{r}^{\mu}}|^q \ dz \rbr^{\frac{1}{q}},\\
                  \|\nabla u\| \geq \frac{\mu}{4} \txt{on $Q_{r/2}^{\mu}$.}
     \end{gather*}
 \end{proposition}

\subsection{Proof of second alternative - \texorpdfstring{\cref{alt2}}.}

Having fixed $\nu$ according to \cref{alt1}, we are now in the situation that \cref{9.48} and \cref{9.48_alt} does not hold. This condition is equivalent to 
\begin{equation}\label{alt2_cond}
        |\{ Q_{r}^{\mu}: u_{x_i} \geq \mu/2\}| < (1-\nu) |Q_r^{\mu}| \txt{and}|\{ Q_{r}^{\mu}: u_{x_i} \leq -\mu/2\}| < (1-\nu) |Q_r^{\mu}|.
\end{equation}
From \cref{sup_u}, as in \cref{3.3}, we again assume the following is always satisfied for some $r \in (0,R]$ and $A \geq 1$:
\begin{equation}\label{3.3_alt2}
    s + \sup_{Q_r^{\mu}} \|\nabla u\| \leq A \mu.
\end{equation}

Let us perform the following change of variables:
    \begin{equation}\label{9.69}
        w(x,t) = \frac{u(\mu^{-1}rx,\mu^{-p}r^2t)}{rA} \txt{for} (x,t) \in Q_1.
    \end{equation}
    Then we have 
    \begin{equation}\label{9.70}
        \frac{s}{\mu A} + \sup_{Q_1} \|\nabla w\| \leq 1 \txt{on} Q_1.
    \end{equation}
    Moreover, \cref{alt2_cond} becomes
    \begin{equation}
        \label{alt2_new}
        \abs{\left\{ Q_{1}: w_{x_i} \geq \frac{1}{2A}\right\}} < (1-\nu) |Q_1| \txt{and}\abs{\left\{ Q_{1}: w_{x_i} \leq -\frac{1}{2A}\right\}} < (1-\nu) |Q_1|.
    \end{equation}

\subsubsection{Choosing a good time slice}
First let us show there exists a good time slice.
\begin{lemma}\label{cht}
    There exists $t_{\ast}$ such that $ -1 \leq t_{\ast} \leq -\frac{\nu}{2}$ and
    \begin{equation}\label{9.68}
        \abs{\left\{ B_{1} : w_{x_i}(x,t_{\ast}) \geq \frac{1}{2A}\right\}} \leq \lbr \frac{1-\nu}{1-\nu/2}\rbr |B_{1}|
    \end{equation}

\end{lemma}

\begin{proof}
    The proof is by contradiction, suppose \cref{9.68} does not hold for any $t \in (-1,-\frac{\nu}{2})$, then we have 
    \begin{equation*}
        \begin{array}{rcl}
        (1-\nu) |Q_1|\overset{\cref{alt2_new}}{>}    \abs{\left\{ Q_1 : w_{x_i} \geq \frac{1}{2A}\right\}} & = & \int_{-1}^{-\frac{\nu}{2}}  \abs{\left\{ B_{1} : w_{x_i}(x,t) \geq \frac{1}{2A}\right\}} \ dt \\
            &\overset{\text{\cref{9.68} fails}}{>} & \lbr (1-\nu)\rbr |B_{1}| ,
        
        \end{array}
    \end{equation*}
     which is a contradiction.
\end{proof}
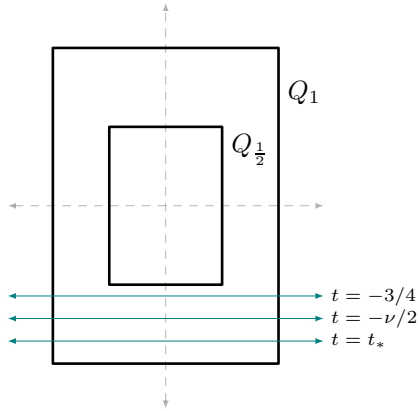
\begin{figure}[ht!]
\begin{center}
\begin{tikzpicture}[line cap=round,line join=round,>=latex,x=0.3cm,y=0.3cm]
\draw [line width=1pt,color=black] (-5,-7)-- (5,-7) -- (5,7) -- (-5,7) -- cycle;
\draw [line width=1pt,color=black] (-2.5,-3.5) -- (2.5,-3.5)-- (2.5,3.5) -- (-2.5,3.5) -- cycle;

\draw [line width=0.1pt,opacity=0.3,dashed,<->] (-7,0) -- (7,0);
\draw [line width=0.1pt,opacity=0.3,dashed,<->] (0,-9) -- (0,9);

\draw [line width=0.1pt,opacity=1,color=teal,<->] (-7,-4) -- (7,-4);
\draw [line width=0.1pt,opacity=1,color=teal,<->] (-7,-5) -- (7,-5);
\draw [line width=0.1pt,opacity=1,color=teal,<->] (-7,-6) -- (7,-6);
\begin{scriptsize}
\draw[color=black] (7,-4) node[right] {$t=-3/4$};
\draw[color=black] (7,-5) node[right] {$t=-\nu/2$};
\draw[color=black] (7,-6) node[right] {$t=t_{\ast}$};
\end{scriptsize}
\draw[color=black] (5,5) node[right] {$Q_1$};
\draw[color=black] (2.5,2.5) node[right] {$Q_{\frac{1}{2}}$};
\end{tikzpicture}
\end{center}
\end{figure}
\subsubsection{Rescaling the equation}
Recalling \cref{9.69}, let us define
    \begin{equation}\label{9.73}
        \hat{\aa}(\zeta) := \frac{1}{\mu^{p-1} A} \aa(A \mu \zeta),
    \end{equation}
        then, we see that $w$ solves 
        \begin{equation}\label{9.74}
            w_t - \dv \hat{\aa}(\nabla w) = 0 \txt{in} Q_1. 
        \end{equation}

\begin{lemma}\label{lemma_elip_resc}
    Differentiating \cref{9.74} with respect to $x_i$ for some $i \in \{1,2,\ldots,N\}$, we get
    \begin{equation}\label{w_x_i}
    (w_{x_i})_t - \dv \pa\hat{\aa}(\nabla w)\nabla w_{x_i} = 0 \txt{in} Q_1. 
    \end{equation}
    Then $\pa\hat{\aa}(\zeta)$ and $\hat{\aa}(\zeta)$ satisfies the following structure conditions:
    \begin{equation}\label{9.75}
        \begin{array}{c}
            |\hat{\aa}(\zeta)| + |\pa\hat{\aa}(\zeta)| \lbr |\zeta|^2 + \lbr\frac{s}{A\mu}\rbr^2 \rbr^{\frac12} \leq {C_1A^{p-2}}   \lbr |\zeta|^2 + \lbr\frac{s}{A\mu}\rbr^2 \rbr^{\frac{p-1}{2}},\\
            C_0A^{p-2}  \lbr |\zeta|^2 + \lbr\frac{s}{A\mu}\rbr^2 \rbr^{\frac{p-2}{2}} |\eta|^2 \leq \iprod{\pa \hat{\aa}(\zeta)\eta}{\eta}.
        \end{array}
    \end{equation}
\end{lemma}
\begin{proof}
    The proof follows directly from \cref{structure_aa_holder} and \cref{9.73} and by using the fact that $\mu \leq \mu_0$. 
\end{proof}

\subsubsection{Logarithmic and Energy estimates}
\begin{lemma}\label{log_estimate}
    Let $0 < \eta_0 < \nu$ and $k \geq \frac{1}{4A}$ be any two fixed numbers (recall $\nu$ is from \cref{alt2}) and consider the function
\begin{equation*}%
    \Psi(z) := \log^+ \lbr \frac{\nu}{\nu - (z-(1-\nu))_+ + \eta_0} \rbr.
\end{equation*}%
Then there exists constant $C = C(N,p,C_0,C_1,A)$ such that for all $ t_1,t_2 \in (-1,1)$ with $t_1<t_2$ and all $s \in (0,1)$, there holds
\begin{equation}\label{log_est}
    \int_{B_{s} \times \{t_2\}} \Psi^2((w_{x_i}-k)_+) \ dx \leq \int_{B_{1} \times \{t_1\}} \Psi^2((w_{x_i}-k)_+) \ dx + \frac{C}{(1-s)^2} \iint_{B_1\times (t_1,t_2)} \Psi((w_{x_i}-k)_+) \ dz. 
\end{equation}
\end{lemma}
\begin{proof}
    Let us take $\Psi((w_{x_i}-k)_+) \Psi'((w_{x_i}-k)_+) \zeta^2$ as a test function in \cref{w_x_i}. Since $k \geq \frac1{4A}$, we see that $(w_{x_i} - k)_+ =0 $ whenever $w_{x_i} \leq \frac1{4A}$. On this set, $\Psi(0) = \log^+ \lbr \frac{\nu}{\nu+\eta_0}\rbr = 0$, thus we see that $\spt{\Psi((w_{x_i}-k)_+)}$ is contained in  the set $\left\{w_{x_i} \geq \frac1{4A}\right\}$, therefore from \eqref{9.75} it follows that    
 \begin{equation}\label{unif_elliptic}
    (w_{x_i})_t - \dv B(x,t) \nabla w_{x_i} = 0, \txt{with}  \tilde C_0|\zeta|^2 \leq \iprod{B(x,t) \zeta}{\zeta} \leq C_1 A^{p-1}  |\zeta|^2.
\end{equation}
where $\tilde C_0= \frac{C_0}{2^{p-2}}$ for $p >2 $ and $\tilde C_0 = C_0  A^{p-2}$ for $p<2$. 
We can now follow the proof of \cite[Proposition 3.2 of Chapter II]{DB93} (see also \cite[(12.7) of Chapter IX]{DB93}) to get 
\[
    \int_{B_{s} \times \{t_2\}} \Psi^2((w_{x_i}-k)_+) \ dx \leq \int_{B_{s} \times \{t_1\}} \Psi^2((w_{x_i}-k)_+) \ dx + \frac{C_{(N,p,C_0,C_1,A)}}{(1-s)^2} \iint_{B_1\times (t_1,t_2)} \Psi((w_{x_i}-k)_+) \ dz. 
\]

\end{proof}


\begin{lemma}
    \label{energy_estimate}
    Let $k \geq \frac1{4A}$ be some fixed constant, then $w$ solving \cref{9.74} satisfies
    \[
    \begin{array}{l}
        \sup_{t_0<t<t_1} \int_{B_1}  (w_{x_i} - k)_+^2 \zeta^2 \ dx  + C \iint_{B_1\times (t_0,t_1)} |\nabla (w_{x_i} - k)_+|^2 \zeta^2 \ dz\leq \int_{B_1 \times \{t=t_0\}} (w_{x_i} - k)_+^2 \zeta^2 \ dx \\
        \hspace*{10cm} + C\iint_{B_1 \times (t_0,t_1)} (w_{x_i} - k)_+^2 |\nabla \zeta|^2  \ dz\\
        \hspace*{10cm} + C\iint_{B_1 \times (t_0,t_1)} (w_{x_i} - k)_+^2 |\zeta| |\zeta_t|  \ dz,
        \end{array}
    \]
where $t_0,t_1 \in [-1,1)$ with $t_0<t_1$ are some fixed time slices and $\zeta \in C^{\infty}(Q_1)$ is a cut-off function. Here $C$ has the same dependence as in Lemma \ref{log_estimate}. 
\end{lemma}
\begin{proof}
    We take $(w_{x_i} - k)_+ \zeta^2$ as a test function in \cref{9.74}, noting that we retain the uniformly elliptic structure of $B(x,t)$ in \cref{unif_elliptic} from the choice of test function. Thus proceeding according to \cite[Proposition 3.1 of Chapter II]{DB93} gives the desired estimate.
\end{proof}

\subsubsection{DeGiorgi type iteration}

Let us now  first prove an expansion of positivity in time result:
\begin{lemma}\label{lemma921}
    There exists $\eta_0 = \eta_0(N,p,\nu,A, \mu_0) \in (0,\nu)$ such that for all $t \in (t_{\ast},1)$ with $t_{\ast}$ as in Lemma \ref{cht}, there holds
    \[
        \abs{\left\{x \in B_1: w_{x_i} (x,t) > (1-\eta_0)  \right\}} \leq \lbr 1-\frac{\nu^2}{4} \rbr |B_1|.
    \]
\end{lemma}
\begin{proof}
    Let us apply \cref{log_est} over the time interval $(t_{\ast},t)$ and estimate each of the terms as follows:
    \begin{description}[leftmargin=*]
        \item[Estimate for the term on the left hand side of  \cref{log_est}:] On the set $\{w_{x_i} > (1-\eta_0)\}$, we see that the log function satisfies $\Psi(w_{x_i}- k)_+) \geq \log \lbr \frac{\nu}{2\eta_0}\rbr$, thus we get
        \begin{equation}
            \label{estimate_1}
            \int_{B_{s} \times \{t_2\}} \Psi^2((w_{x_i}-k)_+) \ dx  \geq \log^2 \lbr \frac{\nu}{2\eta_0}\rbr \abs{\{x \in B_s: w_{x_i}(x,t_2) > (1-\eta_0)\}}
        \end{equation}

        \item[Estimate for the first term on the right hand side of  \cref{log_est}:] Since the log function $\Psi(w_{x_i}- k)_+)$ vanishes on the set  $\{w_{x_i} \leq 1/(4A)\}$, we estimate this term as follows:
        \begin{equation}
            \label{estimate_2}
            \int_{B_1 \times \{t=t_{\ast}\}} \Psi^2(w_{x_i}- k)_+)  \ dx \overset{\cref{9.70}}{\leq} \log^2\lbr \frac{\nu}{\eta_0} \rbr \abs{\{x \in B_1: w_{x_i}(x,t_{\ast}) \geq 1/(4A)\}} \overset{\cref{9.68}}{\leq} \lbr \frac{1-\nu}{1-\nu/2}\rbr\log^2\lbr \frac{\nu}{\eta_0} \rbr |B_1|.
        \end{equation}
        \item[Estimate for the second term on the right hand side of  \cref{log_est}:] Analogously, we estimate this term as follows:
        \begin{equation}
            \label{estimate_3}
            \frac{C}{(1-s)^2} \iint_{B_1\times (t_1,t_2)} \Psi((w_{x_i}-k)_+) \ dz \leq \frac{C}{(1-s)^2} |B_1| \log\lbr \frac{\nu}{\eta_0} \rbr.
        \end{equation}
    \end{description}
    Combining \cref{estimate_1}, \cref{estimate_2} and \cref{estimate_3}, we get the following sequence of estimates:
    \begin{equation*}
        \begin{array}{rcl}
            \abs{\{x \in B_1: w_{x_i}(x,t_2) > (1-\eta_0)\}} & \leq & \abs{\{x \in B_s: w_{x_i}(x,t_2) > (1-\eta_0)\}} + (1-s) |B_1|\\
            & \leq & \lbr \lbr \frac{1-\nu}{1-\nu/2}\rbr \lbr \frac{\log^2 \lbr \frac{\nu}{\eta_0}\rbr}{\log^2 \lbr \frac{\nu}{2\eta_0}\rbr}\rbr+  \frac{C}{(1-s)^2}\lbr \frac{\log \lbr \frac{\nu}{\eta_0}\rbr}{\log^2 \lbr \frac{\nu}{2\eta_0}\rbr}\rbr + (1-s)\rbr|B_1|.
        \end{array}
    \end{equation*}
    Choosing $s$ followed by $\eta_0$ appropriately gives the desired conclusion. 
\end{proof}
Given Lemma  \ref{lemma921}, the rest of the proof of the second alternative Proposition \ref{alt2}   is as in the linear case. We nevertheless provide details for the sake of completeness.

\begin{definition}\label{idef}
    Having determined $\eta_0$ as in \cref{lemma921}, let $j_0$ be the largest positive integer such that $2^{-j_0} \geq \eta_0$. Then for $j \geq j_0$ and any $t \in (t_{\ast},1)$, let us define
    \[
        A_j(t) := \left\{ B_1: w_{x_i} (x,t) > (1-2^{-j})\right\} \txt{and} A_j^s:=\int_{t_{\ast}}^{s} A_j(t) \ dt,
    \]
    for $s \in (t_{\ast},1]$ with $s-t_{\ast} \geq \frac18$.
\end{definition}

\begin{lemma}\label{lemma9.23}
    For every $\de \in (0,1)$, there exists $j_{\de} \geq j_0$ such that for any $s \in (t_{\ast},1]$ with $s-t_{\ast} \geq \frac18$, the following holds: $$A_{j_{\de}}^s \leq \de |B_1| |s-t_{\ast}|.$$
    Note that the estimates and choice of $j_{\de}$ are all independent of the choice of $s$.
\end{lemma}
\begin{proof}
From \cref{lemma921}, we see that 
    \begin{equation}\label{983}
        |B_1 \setminus A_j(t)| \geq \lbr \frac{\nu}{2}\rbr^2 |B_1| \txt{for any} t \in (t_{\ast},1].
    \end{equation}
    Let us apply \cref{Poincare} with the choices $l= 1-\frac{1}{2^{j+1}}$ and $k= 1-\frac{1}{2^j}$ for some $j \geq j_0$ where $j_0$ is as defined in \cref{idef} to get
    \[\begin{array}{rcl}
        \frac{1}{2^{j+1}} |A_{j+1}(t)| &\leq&  \frac{C}{\abs{B_1 \setminus A_j(t)}} \int_{A_j(t) \setminus A_{j+1}(t)} |\nabla w_{x_i}| \ dx \\
        & \overset{\cref{983}}{\leq}&  C_{(N,p,\nu)} \lbr \int_{B_1} |\nabla (w_{x_i} - (1-2^{-j}))_+|^2 \ dx \rbr^{\frac12} |A_j(t) \setminus A_{j+1}(t)|^{\frac12}.    
        \end{array}
    \]
Integrating over $(t_{\ast},s)$, we get
\begin{equation}\label{984}
        \frac{1}{2^{j+1}} |A_{j+1}^s|  \leq C_{(N,p,\nu)} \lbr \iint_{B_1\times (t_{\ast},s)} |\nabla (w_{x_i} - (1-2^{-j}))_+|^2 \ dz \rbr^{\frac12} |A_j^s \setminus A_{j+1}^s|^{\frac12}.   
    \end{equation}
    
    From \cref{9.70} and the choice of $k=1-\frac{1}{2^j}$, we get
    \begin{equation*}%
        (w_{x_i} - k)_+ \leq \frac{1}{2^j},
    \end{equation*}%
    using which, we  now estimate the gradient term on the right hand side of \cref{984} using \cref{energy_estimate} to get
    \begin{equation}\label{985}
        \iint_{B_1\times (t_{\ast},s)} |\nabla (w_{x_i} - (1-2^{-j}))_+|^2 \ dz \leq C \frac{1}{4^j} |B_1| |s-t_{\ast}|,
    \end{equation}
    where we used the fact that $|\zeta_t| \leq \frac{C}{s-t_{\ast}} \leq C$. Substituting \cref{985} into \cref{984} and summing over $j=j_0,j_0+1, \ldots,j_{\de}$, we get
    \begin{equation*}%
        (j_{\de} - j_0) |A_{j_{\de}}^s|^2 \leq C |B_1| |s-t_{\ast}| \sum_{j=j_{\de}}^{j_0} |A_j^s \setminus A_{j+1}^s| \leq C\lbr |B_1| |s-t_{\ast}|\rbr^2.
    \end{equation*}%
Choosing $j_{\de}$ sufficiently large followed by taking square roots gives the desired  conclusion.

\end{proof}

\subsubsection{Combining all the estimates to prove \texorpdfstring{\cref{alt2}}.}

Let us prove the following result by applying DeGiorgi iteration, using which we  easily conclude \cref{alt2}.
\begin{proposition}\label{main_prop}
    Suppose the first (analogously second) inequality in \cref{alt2_new} holds, then there exist positive constant $\eta  \in (0,1)$ such that the following conclusion holds:
    \[
        \abs{\left\{Q_{\frac12}: w_{x_i}(z) > (1-\eta)\right\}} = 0 \qquad \lbr\text{analogously} \ \ \  \abs{\left\{Q_{\frac12}: w_{x_i}(z) < -(1-\eta)\right\}} = 0\rbr.
    \]
\end{proposition}

\begin{proof}
    Consider the family of nested cylinders 
    \[
        Q_n := B_{\rho_n} \times (-\rho_n^2,\rho_n^2) \txt{where} \rho_n:= \frac12 + \frac{1}{2^{n+2}} \quad \text{with} \ n=0,1,2,\ldots,
    \]
    and increasing intervals 
    \[
        k_n:= 1- \frac{1}{2^{j_{\ast}+2}} - \frac{1}{2^{j_{\ast}+2+n}},
    \]
    where $j_{\ast}$ is a number to be eventually chosen. Furthermore, define
    \[
        Y_n:= \abs{\{(x, t) \in Q_n: w_{x_i} (x, t)  > k_n\}}.
    \]
    From \cref{energy_estimate} and \cref{par_sob_emb} applied with $\tp=2$, we get the following sequence of estimates (note that $\zeta_n$ is the usual cut-off function with $\zeta_n= 1$ on $Q_{n+1}$ and $\spt(\zeta_n) \subset Q_n$):
    \begin{equation}\label{988}
        \begin{array}{rcl}
            \iint_{Q_{n+1}} (w_{x_i} - k_n)_+^2 \ dz & \overset{\redlabel{988a}{a}}{\leq} &  \iint_{Q_{n}} (w_{x_i} - k_n)_+^2\zeta_n^2 \ dz \\
            & \overset{\redlabel{988b}{b}}{\leq} & \lbr \iint_{Q_n} | (w_{x_i} - k_n)_+\zeta_n|^{\frac{2N}{N+2}} \ dz \rbr^{\frac{N}{N+2}} Y_n^{\frac{2}{N+2}}\\
            &  \overset{\redlabel{988c}{c}}{\apprle}  & \|(w_{x_i} - k_n)_+^2\zeta_n^2\|_{V^{2,2}(Q_n)}^2 Y_n^{\frac{2}{N+2}}\\
            &  \overset{\redlabel{988d}{d}}{\leq}  & C 4^{j_{\ast}} Y_n^{1+\frac{2}{N+2}},
        \end{array}
    \end{equation}
    where to obtain \redref{988a}{a}, we enlarged the domain of integration noting that $\zeta_n \equiv 1$ on $Q_{n+1}$, to obtain \redref{988b}{b}, we applied H\"older inequality and made use of the definition of $Y_n$, to obtain \redref{988c}{c}, we applied \cref{par_sob_emb} with $\tp=2$ and $\tq = \frac{2(N+2)}{N}$ and finally to obtain \redref{988d}{d}, we made use of the definition of $k_n$ along with \cref{energy_estimate}.

On the other hand, from \cref{chebyschev}, we also have
\begin{equation}\label{989}
    Y_{n+1} \leq C 4^{n+j_{\ast}} \iint_{Q_{n+1}} (w_{x_i} - k_n)_+^2 \ dz.
    \end{equation}
    Combining \cref{989} and \cref{988}, we get 
    \[
        Y_{n+1} \leq C 4^{n} Y_n^{1+\frac{2}{N+2}},
    \]
    and applying \cref{iteration}, we see that if $$Y_0 \leq C^{-\frac{N+2}{2}}4^{-\lbr \frac{N+2}{2}\rbr^2} := \de,$$ then $Y_n \rightarrow 0$ as $n \rightarrow \infty$. 

    Note that $\de = \de(\data)$, thus with this choice of $\de$, we can obtain $j_{\de}$ (and set $j_{\ast} = j_{\de}$) from \cref{lemma9.23} such that the condition for $Y_0$ is satisfied and thus $Y_{\infty} = 0$ which is the desired conclusion with $\eta = 2^{-(j_{\ast} +2)}$.
\end{proof}

\section*{References}
\bibliographystyle{plain}

\end{document}